\documentclass[a4paper]{amsart}

\newtheorem{theorem}{Theorem}[section]
\newtheorem{lemma}[theorem]{Lemma}

\theoremstyle{definition}
\newtheorem{definition}[theorem]{Definition}

\theoremstyle{remark}
\newtheorem{remark}[theorem]{Remark}

\numberwithin{equation}{section}
\usepackage{graphicx}              
\usepackage{amsmath}               
\usepackage{amsfonts}              
\usepackage{amsthm}
\usepackage{amsmath}
\usepackage{stmaryrd} 
\usepackage{chemarr, graphicx,  subfigure} 
\usepackage[usenames, dvipsnames]{color}
\usepackage{pict2e}
\usepackage{picture}

\usepackage{comment}
\usepackage{grffile}
\usepackage{xcolor}

\usepackage[normalem]{ulem}

\newcommand{\Rev}[1]{{#1}}

\newcommand{\vertiii}[1]{{\|\kern-0.25ex | #1 
    | \kern-0.25ex \|}}

\newcommand{\dint}{\text{\rm int}}

\newcommand{\mean}[1]{\{\kern-1.1mm\{#1\}\kern-1.1mm\}}          
\newcommand{\wmean}[1]{\mean{#1}_{\mbf{w}}}                 
\newcommand{\jump}[1]{[\![#1]\!]}                        

\newcommand{\ltwo}[2]{\|{#1}\|_{#2}}

\newcommand{\ud}{\,\mathrm{d}}
\newcommand{\ndg}[1]{| \kern -.25mm \|{#1}| \kern -.25mm \|}
\newcommand{\nsdg}[1]{| \kern -.25mm \|{#1}| \kern -.25mm \|_{\rm s}}

\newcommand{\vecp}{{\bf p}}

\newcommand{\fes}{S_\mesh^{\vecp} }
\newcommand{\fesmone}{S_\mesh^{\vecp-{\bf 1}} }

\newcommand{\fesmapped}{S_{\mesh,map}^{\vecp} }

\newcommand{\diam}{\operatorname{diam}}

\newcommand{\mesh}{\mathcal{T}}

\newcommand{\norm}[2]{\|{#1}\|_{{#2}}}
\newcommand{\ncdg}[1]{| \kern -.25mm \|{#1}| \kern -.25mm \|_{\rm DG}}

\renewcommand{\k}{K}
\renewcommand{\mesh}{\mathcal{T}}
\newcommand{\mbf}[1]{\mathbf{#1}}

\newcommand{\bnabla}{\nabla_{\kern-.08cm\mesh}^{}}

\title[Robust interior penalty discontinuous Galerkin methods]{Robust interior penalty\\ discontinuous Galerkin methods}

\author{Zhaonan Dong} \address{
	1) Inria, 2 rue Simone Iff, 75589 Paris, France, 
	and 2) CERMICS, Ecole des Ponts, 77455 Marne-la-Vall\'{e}e 2, France }  \email{Zhaonan.Dong@inria.fr}
\author{Emmanuil H.~Georgoulis} \address{
1) School of Computing and Mathematical Sciences,
University of Leicester,
University Road,
Leicester, LE1 7RH,
United Kingdom, 
 2) Department of Mathematics, School of Applied Mathematical and Physical Sciences, National Technical University of Athens, Zografou 15780, Greece, and
3) IACM-FORTH, Crete, Greece} \email{Emmanuil.Georgoulis@le.ac.uk}

\begin{document}
	
	\begin{abstract} 
	 A new variant of the IPDG method is presented, involving carefully constructed weighted averages of the gradient of the approximate solution. The method is shown to be robust even for the most extreme simultaneous local mesh, polynomial degree and diffusion coefficient variation scenarios, without resulting into unreasonably large penalization. The new IPDG method, henceforth termed as \emph{robust IPDG} (RIPDG),  offers typically significantly better error behaviour and conditioning than the standard IPDG method when applied to scenarios with strong mesh/polynomial degree/diffusion local variation, \Rev{especially when} the underlying approximation space does not contain a sufficiently rich conforming subspace. The latter is of particular importance in the context of IPDG methods on polygonal/polyhedral meshes.  On the other hand, when using uniform meshes, constant polynomial degree for problems with constant diffusion coefficients the RIPDG method is identical to the classical IPDG.  Numerical experiments indicate the favourable performance of the new RIPDG method over the classical version in terms of conditioning and error.
	\end{abstract}

	\maketitle
	
	\section{Introduction}
	Discontinuous Galerkin  (DG) methods are a staple in the cascade of Galerkin approaches for the numerical solution of PDEs in divergence form. For the discretization of elliptic operators, the class of interior penalty DG (IPDG) methods is arguably the most popular choice due to its versatility and simplicity in implementation. The consistency of the IPDG method is achieved by the inclusion of specific terms involving integrals of solution fluxes on element interfaces. The basic, yet revolutionary, idea in IPDG methods is the further addition of penalization terms involving the jump of the approximate solution (or its derivatives for higher order elliptic PDEs) to control the departure from the PDE solution space and, therefore, to achieve stability in a consistent fashion. The IPDG method design concept was presented in full by Baker \cite{baker} and, briefly later, by Wheeler \cite{wheeler:78} and Arnold \cite{arnold}, thereby providing a consistent extension to the classical work of Nitsche \cite{nitsche1971variationsprinzip} for the weak imposition of boundary conditions \Rev{and} to the classical inconsistent finite element method with penalty of Babu{\v s}ka \cite{Babuska73}. These classical IPDG methods and their $hp$-version variants \cite{rwg,newpaper} are widely used to this day.

	Classical interior penalty discontinuous Galerkin (IPDG) methods for diffusion problems require a number of assumptions on the local variation of mesh-size, polynomial degree, and of the diffusion coefficient to determine the values of the, so-called, discontinuity-penalization parameter and/or to perform error analysis. Variants of IPDG methods involving weighted averages of the normal flux functions of the approximate solution have been proposed \cite{dryja,burman_zunino,ern_weighted} in the context of high-contrast diffusion coefficients to mitigate the dependence of the contrast in the stability and in the error analysis. More recently, weighted averages involving local meshsizes have been proposed for treating interface problems \cite{Robust_Nitsche,MR4078806} on non-matching grids on either side of the interface.
	

	To fix ideas, we shall focus on the basic second order elliptic boundary value problem on an open domain $\Omega\subset\mathbb{R}^d$, $d\in \mathbb{N}$: find $u\in H^1(\Omega)$ such that
	\begin{equation}\label{pde}
		\begin{aligned}
		-\nabla\cdot (a\nabla u) =&\  f \quad \text{ in } \Omega;\\
			 u=&\  g \quad \text{ on } \partial\Omega,
	\end{aligned}
	\end{equation}
with data $f\in L_2(\Omega)$, $g\in L_2(\partial\Omega)$ and a positive definite diffusion tensor $a\in [L_\infty(\Omega)]^{d\times d}$. Although considerable extensions to the problem can be incorporated immediately to the developments presented below, we refrain from exhausting the generalization possibilities in the interest of clarity of exposition.

Classical IPDG methods for the problem \eqref{pde} involve a user-defined quantity, the so-called \emph{penalty} or \emph{discontinuity-penalization} parameter. The judicious choice of this parameter is crucial for the stability of the IPDG method, but also for its practical behaviour regarding approximation quality and conditioning. In particular, it is well known that excessive penalty parameter typically leads to increased ill-conditioning of the IPDG stiffness matrix. At the same time, the penalty parameter has to be chosen large enough to ensure stability of the method. 

In this work, we focus on the  in the behaviour of interior penalty methods in the presence of highly locally variable meshes, local polynomial degrees, and/or diffusion coefficients $a$. As we will see below, in such extreme, often combined, local variation scenarios of discretization parameters and/or coefficients, the classical IPDG method typically requires excessively large penalization to retain stability. In fact, \emph{it is possible that classical IPDG penalty parameter choices degenerate to infinity while the dimension of the approximation spaces remains finite}. This, in turn may be detrimental to the conditioning of the problem and, in certain cases, to the quality of the approximation.

To resolve this state of affairs, we present a new weighted-type interior penalty discontinuous Galerkin method which is provably stable with a new choice of the discontinuity-penalization parameter and is \emph{robust} with respect to extreme local variations in the discretization parameters and in the diffusion coefficient. The method, termed henceforth as \emph{robust IPDG} (RIPDG) replaces the classical average-of-the-normal-flux term(s) in the IPDG by suitable, explicitly constructed weighted averages. This modification results into a robust and concrete selection of the discontinuity-penalization parameter, without resulting into excessive penalization in extreme combined local mesh, polynomial degree and coefficient variation scenarios. We note that a balanced choice for the discontinuity-penalization parameter is particularly important when \emph{no} conforming subspace is available. Moreover, the new RIPDG method allows for the first time in the interior penalty literature for the proof of a priori error bounds \emph{without} any local-quasi-uniformity/local bounded variation assumptions for the discretization parameters.

The modifications required to implement RIPDG starting from an available IPDG code are minimal and trivial. We envisage that the RIPDG method will be effective in the numerical approximation of complex, multiscale problems characterized by extreme local physical features necessitating highly non-uniform Galerkin approximation spaces.  In particular, as we shall see in the numerical experiments below, RIPDG is particularly pertinent in scenarios \Rev{where} \emph{no} conforming subspace of sufficient approximation capabilities is available. This is often the case when employing IPDG methods on meshes consisting of general polygonal/polyhedral (i.e., polytopic) elements, typically with many faces per element. These meshes can arise, among other processes, via mesh agglomeration procedures \cite{cangiani2013hp,book,dg-ease,BassiJCP2012}. Using elements resulting from agglomeration of finer unstructured meshes, in an effort to reduce computational complexity leads naturally to Galerkin spaces with inherent significant non-conformity. In such cases, the effect of discontinuity-penalization becomes evident and, in some extreme cases, appears to be dominating the convergence behaviour (cf., Example 1 with the ``zigzag'' mesh below).

At the other end of the spectrum, when using uniform meshes, constant polynomial degree to solve problems with constant diffusion coefficients RIPDG is identical to the classical IPDG. Moreover, within the usual setting of locally quasi-uniform meshes, even with locally bounded variable polynomial degree and small/medium contrast in the diffusion coefficient, RIPDG and IPDG offer essentially identical numerical solutions.

Indeed, numerical experiments indicate the favourable performance of the new RIPDG method over the classical version in terms of conditioning and error. In particular, despite extensive testing, we have \emph{not} been able to identify an example whereby the RIPDG method is inferior in terms of conditioning and/or error in various norms when compared to the classical IPDG method. In most cases the difference between IPDG and RIPDG is not significant, due to the fact that a sufficiently \Rev{rich} conforming subspace is present in the approximation. However, in certain extreme scenarios RIPDG offers significantly better approximation and/or conditioning. Although we expect that cases with IPDG outperforming RIPDG exist, we have not yet been able to identify such example despite extensive testing.

The remainder of this work is structured as follows. In Section \ref{sec:IPDG} we review the classical IPDG method and the weighted variant from \Rev{\cite{dryja,burman_zunino,ern_weighted}}, along with the basic argument for proving coercivity and continuity of the respective bilinear form. In Section \ref{RIPDG}, we present the new RIPDG method and discuss its construction and properties, as well as the motivating differences in the choice of the penalty parameter. In Sections \ref{poly_RIPDG} and \ref{RIPDG-deg} we provide extensions of RIPDG to polytopic meshes and to degenerate elliptic problems, respectively. In Section \ref{sec:apriori}, we present the basic Strang's-type Lemma satisfied by IPDG and RIPDG, showcasing that in the latter case the resulting a priori error bound is \emph{independent} of any local variation in meshsize, polynomial degree and diffusion contrast. We conclude with a series of numerical experiments showcasing the benefits in using weighted averages in scenarios with high such local variations.
	
\section{Interior penalty discontinuous Galerkin methods}\label{sec:IPDG}
	We consider meshes $\mesh$ consisting of mutually disjoint open simplicial and/or box-type elements $\k\in\mesh$ so that $\cup_{\k\in\mesh}\bar{\k} =\bar{\Omega}$. Let also $h_{\k}:=\diam(\k)$ the diameter of  $\k\in\mesh$, and define the mesh-function $\mbf{h}:\cup_{\k\in\mesh}\k\to\mathbb{R}_+$ by $\mbf{h}|_{\k}=h_{\k}$, $\k\in\mesh$. Let also $m_{\k}$ denote the number of $(d-1)$-dimensional faces of the element $\k\in\mesh$, i.e., $m_\k=d+1$ for a  $d$-dimensional simplex and \Rev{$m_\k = 2d$} for a $d$-dimensional box-type element, and define $\mbf{m}:\cup_{\k\in\mesh}\k\to\mathbb{R}_+$ by $\mbf{m}|_{\k}=m_{\k}$, $\k\in\mesh$.
	Further, we let $\Gamma:=\cup_{\k\in\mesh}\partial\k$ denote the mesh skeleton and set $\Gamma_{\dint}:=\Gamma\backslash\partial\Omega$. The mesh skeleton $\Gamma$ is decomposed into $(d-1)$--dimensional simplicial \Rev{or quadrilateral} \emph{faces} $F$, shared by at most two elements. These are distinct from elemental \emph{interfaces} $I$, herein defined as the simply-connected components of the intersection between the boundaries of two neighbouring elements.
	As such, hanging nodes/edges are naturally permitted: an interface between two elements may consist of more than one \Rev{faces residing on the same hyperplane in the presence of hanging nodes. Moreover, as we will see below, we may wish to view each quadrilateral (inter)face of a three-dimensional box-type element mesh as the union of two simplicial faces.} Let also $\bnabla$ denote the \emph{broken gradient} defined as $\bnabla v|_\k := \nabla (v|_\k)$, $\k\in\mesh$.

	The \emph{discontinuous Galerkin space} $\fes$ with 
	respect to $\mesh$ is defined as
	\[
	\fes :=\{v\in L_2(\Omega)
	:v|_{\k}\in\mathcal{P}_{p_\k^{}}(\k),\,\k\in\mesh\},
	\]
	with
	$\mathcal{P}_{p}(\k)$ denoting the space of $d$-variate polynomials of total degree up to $p$ on $\k$, and $\mbf{p}:\cup_{\k\in\mesh}\k\to\mathbb{R}_+$ with $\mbf{p}|_{\k}=p_{\k}^{}$, $\k\in\mesh$. The local elemental polynomial spaces employed within $\fes$ are defined in the \emph{physical coordinate system}, i.e., without mapping from a given reference or canonical frame (cf.~ \cite{cangiani2013hp,book,dg-ease}). All developments discussed below are also valid for the more `classical' case of the discontinuous Galerkin spaces constructed through the usual mapping $\mbf{F}_\k:\hat{\k}\to\k$ from a suitable reference element $\hat{\k}$, viz.,
		\[
	\fesmapped :=\{v\in L_2(\Omega)
	:v|_\k\circ \mbf{F}_{\k}\in\mathcal{P}_{p_\k^{}}(\hat{\k}),\,\k\in\mesh\}.
	\]
	
	Let $\k_+$ and $\k_-$ be two 
	adjacent elements of $\mathcal{T}$ sharing a face $F\subset\partial \k_+ \cap \partial \k_- \subset \Gamma_{\dint}$. 
	For $v$ and $\mbf{q}$ element-wise continuous scalar- and vector-valued functions, respectively,  we define the \emph{weighted average} across $F$ by
	$$
	\mean{v}_{w_F^{}}|_F:=w_+v_+|_F+w_-v_-|_F, \quad\mean{\mbf{q}}_{w_F^{}}|_F:=w_+\mbf{q}_+|_F+w_-\mbf{q}_-|_F,$$
	for $w_F:=(w_+,w_-)\in \mathbb{R}_+^2$ with $w_+ + w_-=1$,
	respectively, with $v_\pm|_F$ denoting the trace of $v$ from the element $\k_\pm$ on $F$, and correspondingly for $\mbf{q}$. Also, we define the \emph{jump} across $F$ by
	$$
	\jump{v}|_F :=v_+\mbf{n}_+ +v_-\mbf{n}_- , \quad \jump{\mbf{q}}|_F := \mbf{q}_+\cdot\mbf{n}_+ +\mbf{q}_-\cdot\mbf{n}_-,
	$$
with $\mbf{n}_\pm$ denoting the unit outward normal of $\k_\pm$ on $F$. We collect all weight pairs in a \emph{weight function} $\mbf{w}:\Gamma_\dint\to\mathbb{R}_+^2$, with $\mbf{w}|_F=w_F^{}$, for each face $F\subset \Gamma_\dint$.
	On a boundary face $F\subset  \partial\Omega\cap\partial\k$, 
	we set 
	$
	\mean{v}:=v, $  $ \mean{\mbf{q}}:=\mbf{q},  $
	$\jump{v} :=v\mbf{n} ,$ and $ \jump{\mbf{q}} := \mbf{q}\cdot\mbf{n} $, respectively, with $\mbf{n}$ the unit outward normal to $ \partial\Omega$. In view of the latter, we extend $\mbf{w}$ to $\Gamma$ by setting $\mbf{w}|_{\partial\Omega}=(1,0)$ with $\k_-=\emptyset$ there.

The (weighted) interior penalty discontinuous Galerkin method reads: find $u_h\in\fes$, (or in $\fesmapped$,) such that
	\begin{equation}\label{wip}
		B(u_h,v_h)=\ell(v_h),\qquad\text{for all } v_h\in \fes,
	\end{equation}
with 
\[
\begin{aligned}
	B(u_h,v_h):=&\ \int_\Omega a\bnabla u_h\cdot \bnabla v_h\ud x +\int_\Gamma \sigma \jump{ u_h}\cdot \jump{v_h}\ud s \\
	&-\int_\Gamma  \big(\wmean{a\nabla u_h}\cdot \jump{v_h}+\theta \wmean{a\nabla v_h}\cdot \jump{u_h}\big)\ud s,
\end{aligned}
\]
and
\[
\begin{aligned}
	\ell(v_h):=&\ \int_\Omega f v_h\ud x +\int_{\partial\Omega} g(\sigma v_h-\theta a\nabla v_h\cdot\mbf{n}) \ud s ,
\end{aligned}
\]
with $\theta\in[-1,1]$ and $\sigma:\Gamma\to\mathbb{R}$ the, so-called, \emph{discontinuity-penalization} or \emph{penalty} function, whose precise definition, along with that of the weights $\mbf{w}$ is a central concern in this work. Note that, selecting $w_F^{}=(1/2,1/2)$, for all $F\subset\Gamma_\dint$, we retrieve the classical IPDG method of Wheeler \cite{wheeler:78} and of Arnold \cite{arnold}, while the case $w_F^{}=(1,0)$ or $w_F^{}=(0,1)$ for all $F\subset\Gamma_\dint$ gives the original method of Baker \cite{baker}. Moreover, in \cite{dryja,burman_zunino},  the weighted average choice $w_F^{}=(\delta_+,\delta_-)$ with
\[
\delta_\pm := \frac{\alpha_\mp}{\alpha_++\alpha_-},
\]
with $\alpha_\pm:=a|_{\k_\pm}$,
 was proposed for the case of element-wise constant scalar diffusion coefficients, while in \cite{ern_weighted} the selection  $\alpha_\pm:=\mbf{n}_F^Ta|_{\k_\pm}\mbf{n}_F$
 for element-wise constant diffusion tensors was provided, by combining the ideas from \cite{dryja,burman_zunino} and \cite{georgoulis2006note}. More recently, weighted averages involving also local meshsizes have been proposed for treating interface problems \cite{Robust_Nitsche,MR4078806} on non-matching grids on either side of the interface. These choices result to robust \emph{a priori} error analysis with respect to the relative sizes of diffusion locally and/or the relative meshsize across interfaces. The main contribution of this work is the proposition of a different recipe for $\mbf{w}$, which will yield robust dependence not only with respect to the local variation of the diffusion, but also with respect to the local \emph{direction-wise} variations in meshsize and polynomial degree. This will, in turn, yield a robust and rigorous choice for the discontinuity-penalization parameter $\sigma$.

The choice $\theta=1$ in \eqref{wip} yields the \emph{symmetric} version, while the choice $\theta=-1$ the non-symmetric version of the weighted IPDG method. Although, we shall focus on the symmetric version in this work, as the most popular choice, the developments presented below are also valid for the whole range $\theta\in[-1,1]$, with trivial modifications of the arguments and of the formulas.  

To highlight the importance in the careful selection of $\sigma$, we consider a $d$-dimensional simplicial triangulation $\mesh$, we set 
$\theta=1$, $w_F=(1/2,1/2)$ for all faces $F\subset\Gamma_\dint$,  and we study the coercivity of the bilinear form $B(\cdot,\cdot)$ on $\fes$. 
To that end, for $v_h\in \fes$, have
\begin{equation}\label{coer}
\begin{aligned}
	B(v_h,v_h)=&\ \ltwo{\sqrt{a}\bnabla v_h}{L_2(\Omega)}^2 +\ltwo{\sqrt{\sigma} \jump{ v_h}}{L_2(\Gamma) }^2 -2\int_\Gamma  \mean{a\nabla v_h}\cdot \jump{v_h}\ud s,
\end{aligned}
\end{equation}
with $\mean{\cdot}:=\mean{\cdot}_{(1/2,1/2)}$ the usual arithmetic average used in standard IPDG methods. To further estimate the last term on the \Rev{right-hand side of \eqref{coer}}, we employ a
 \emph{trace inverse estimate}, such as the one below.
\begin{lemma}\label{improved_inv}
	Let $\k\in\mesh$ a simplicial $d$-dimensional simplex and one of its faces $F\subset \partial \k$. For any $v \in\mathcal{P}_p(\k)$, we have
	\begin{equation}\label{inv_ineq_gen}
		\ltwo{v}{L_2(F)}^2\le \frac{(p+1)(p+d)}{d}\frac{|F|}{|K|}\ltwo{v}{L_2(\k)}^2.
	\end{equation}
\end{lemma} 
\begin{proof}
	This follows immediately by combining a careful scaling argument with the main result from \cite{warburton2003constants}.
\end{proof}
We refer to \cite{dg-ease} for a generalization of the above estimate to general curved polygonal/polyhedral elements with arbitrary number of faces. Note that for shape-regular elements, \eqref{inv_ineq_gen} yields the familiar version
$
\ltwo{v}{L_2(\partial \k)}^2\le c_{\rm inv} p^2/h_\k \ltwo{v}{L_2(\k)}^2,
$ with the constant $c_{\rm inv}>0$ depending on the aspect ratio of the element $\k$ and the dimension $d$. If, instead, we have $v_h\in\fesmapped$, then $c_{\rm inv}$ will also depend on the elemental maps $F_\k$ if the latter are non-affine. 

Returning to \eqref{coer}, elementary calculations, along with \eqref{inv_ineq_gen}, for faces $F\subset \partial \k_+\cap
\partial \k_-$  give 
\[
\begin{aligned}
& \Big|2\int_\Gamma  \mean{a\nabla u_h}\cdot \jump{v_h}\ud s \Big|\\
\le &
\Rev{\sum_{F\subset \Gamma}\sum_{*\in\{+,-\}} \int_F  |a\nabla u_h|_{\k_*}  \jump{v_h}|\ud s}\\
 \le & \sum_{F\subset \Gamma}\sum_{*\in\{+,-\}} \!\!\!C_{\rm inv,*}\|a\mbf{n}|_{\k_*}\|_{L_\infty(F)}\ltwo{\nabla u_h}{L_2(\k_*)} \ltwo{\jump{v_h}}{L_2(F)}\\
 	\le & \sum_{F\subset \Gamma}\sum_{*\in\{+,-\}} \!\!\! C_{{\rm inv},*}\|a\mbf{n}|_{\k_*}\|_{L_\infty(F)}\|a^{-\frac{1}{2}}\|_{L_\infty(\k_*)}\ltwo{\sqrt{a}\nabla u_h}{L_2(\k_*)} \ltwo{\jump{v_h}}{L_2(F)}\\
 \le &\ \frac{1}{2} \ltwo{\sqrt{a}\bnabla u_h}{L_2(\Omega)}^2+\frac{1}{2}\sum_{F\subset \Gamma}\tau_F\ltwo{\jump{v_h}}{L_2(F)}^2,
 \end{aligned}
 \]
with  $C_{{\rm inv},*}:=\sqrt{p_{\k_*}^{} (p_{\k_*}^{}\!\!\!+d\!-\!1)|F|/(d|\k_*|)}$,  
\begin{equation}\label{tau}
\tau_F:=2\max_{*\in\{+,-\}} m_{\k_*}^{}C_{{\rm inv},*}^2\|a\mbf{n}|_{\k_*}\|_{L_\infty(F)}^2\|a^{-1}\|_{L_\infty(\k_*)},
\end{equation}
 and $m_{\k_*}$ the number of faces of the element $\k_*$, defined above; note that $\bnabla u_h\in \fesmone$, i.e., the discontinuous Galerkin space defined by the same mesh having polynomial basis of one degree less than $\fes$.

These developments, in conjunction with \eqref{coer}, imply
\[
 	\begin{aligned}
 		B(v_h,v_h)\ge&\ \frac{1}{2}\ltwo{\sqrt{a}\bnabla u_h}{L_2(\Omega)}^2 +\sum_{F\subset \Gamma}\int_F \big(\sigma-\frac{\tau_F}{2}\big) \jump{ u_h}^2\ud s,
 	\end{aligned}
 \]
 upon setting $v_h=u_h$.
 Therefore, the selection $\sigma|_F = \tau_F$, for all faces $F\subset \Gamma$ gives the coercivity estimate
 \begin{equation}\label{coer2}
		B(v_h,v_h)\ge \frac{1}{2} \ndg{v_h}^2,\ \text{ with }\ 
	 \ndg{v_h}:=	\sqrt{\ltwo{\sqrt{a}\bnabla u_h}{L_2(\Omega)}^2 + \ltwo{\sqrt{\sigma}\jump{v_h}}{L_2(\Gamma)}^2}\,.
\end{equation}
The above choice of the discontinuity-penalization parameter $\sigma$ is, to the best of our knowledge, the sharpest general estimate for simplicial meshes. An analogous choice of $\sigma$ was proposed in \cite{cangiani2013hp,dg-ease} for significantly more complex element shapes. We note that hanging nodes are permissible by viewing a simplex as a polytopic element with many co-planar faces and by modifying $m_\k$ accordingly. \Rev{The case of box-type elements will be discussed in more detail in Section \ref{sec:box} below.}

Let us define now \Rev{an} extension of the  bilinear form $B:\fes\times\fes\to\mathbb{R}$ for the special case at hand, to accept arguments from the larger space $\mathcal{V}:=H^1(\Omega)+\fes$, viz.,  $B:\mathcal{V}\times\mathcal{V}\to \mathbb{R}$ with
\[
\begin{aligned}
	B(\Rev{z},v):=&\ \int_\Omega a\bnabla \Rev{z}\cdot \bnabla v\ud x +\int_\Gamma \sigma \jump{\Rev{z}}\cdot \jump{v}\ud s \\
&-\int_\Gamma  \big(\mean{a\Pi_{\bf{p-1}} \nabla \Rev{z}}\cdot \jump{v}+ \mean{a\Pi_{\bf{p-1}} \nabla v}\cdot \jump{\Rev{z}}\big)\ud s,
\end{aligned}
\]	
for all $\Rev{z},v\in\mathcal{V}$, with $\Pi_{\bf{p-1}}:L_2(\Omega)\to \fesmone$ denoting the orthogonal $L_2$-projection operator onto $\fesmone$. Completely analogous arguments to the ones presented for the proof of \eqref{coer2}, along with the stability of $\Pi_{\bf{p-1}}$ in the $L_2$-norm, result in the continuity bound
  \begin{equation}\label{cont}
 	B(\Rev{z},v)\le \frac{3}{2} \ndg{\Rev{z}}\ndg{v},\qquad \Rev{z},v\in\mathcal{V}.
 \end{equation}

\begin{remark}
If a different Galerkin space than $\fes$ is used, one has to modify the coercivity and continuity analysis and, correspondingly, the definition of $\sigma$ accordingly. In particular, it may not be the case anymore that $\bnabla u_h\in\fesmone$, but rather that $\bnabla u_h$ resides on the Galerkin space itself.
\end{remark}

\section{A robust interior penalty discontinuous Galerkin method}\label{RIPDG}

We now propose a  \emph{robust} interior penalty discontinuous Galerkin method by taking advantage of appropriately selected weights in \eqref{wip}.  In particular, we will select diffusion tensor, mesh and local polynomial degree dependent weights. To that end, upon considering \eqref{wip}, we have
\begin{equation}\label{coer_wip}
	\begin{aligned}
		B(v_h,v_h)=&\ \ltwo{\sqrt{a}\bnabla v_h}{L_2(\Omega)}^2 +\ltwo{\sqrt{\sigma} \jump{ v_h}}{L_2(\Gamma) }^2 -2\int_\Gamma  \wmean{a\nabla v_h}\cdot \jump{v_h}\ud s.
	\end{aligned}
\end{equation}
Setting $C_{{\rm inv},*}:=\sqrt{p_{\k_*}^{} (p_{\k_*}^{}\!\!\!+d\!-\!1)|F|/(d|\k_*|)}$, using \eqref{inv_ineq_gen}, and working as before, we get, for $F\subset \partial \k_+\cap
\partial \k_-$,
\[
\begin{aligned}
&	\Big|2\int_\Gamma  \wmean{a\nabla u_h}\cdot \jump{v_h}\ud s \Big|\\
\le& 
	\ 2\! 	\sum_{F\subset \Gamma}\int_F \Big(\sum_{*\in\{+,-\}} w_*\|a\mbf{n}|_{\k_*}\|_{L_\infty(F)}|\nabla u_h|_{\k_*}| \Big)| \jump{v_h}|\ud s\\
	\le & 
		\ 2\! \sum_{F\subset \Gamma}\sum_{*\in\{+,-\}} \!\!\! C_{{\rm inv},*}w_*\|a\mbf{n}|_{\k_*}\|_{L_\infty(F)}\|a^{-\frac{1}{2}}\|_{L_\infty(\k_*)}\ltwo{\sqrt{a}\nabla u_h}{L_2(\k_*)} \ltwo{\jump{v_h}}{L_2(F)}.
\end{aligned}
	\]
	Setting now 
	\[
	\zeta_*:=\Big(2\sqrt{m_{\k_*}}C_{{\rm inv},*}\|a\mbf{n}|_{\k_*}\|_{L_\infty(F)}\|a^{-\frac{1}{2}}\|_{L_\infty(\k_*)}\Big)^{-1}, 
	\]
	for $*\in\{+,-\}$, and selecting 
	\begin{equation}\label{new_weights}
	w_* =\frac{\zeta_*}{\zeta_++\zeta_-},
	\end{equation}
	we deduce
	\[
	\begin{aligned}
	\Big|2\int_\Gamma  \wmean{a\nabla u_h}\cdot \jump{v_h}\ud s \Big|	\le &\ \frac{1}{2} \ltwo{\sqrt{a}\bnabla u_h}{L_2(\Omega)}^2+\frac{1}{2}\sum_{F\subset \Gamma}(\zeta_++\zeta_-)^{-2}\ltwo{\jump{v_h}}{L_2(F)}^2.
\end{aligned}
\]
A reasonable selection for the discontinuity-penalization parameter is, therefore, 
\begin{equation}\label{robust_penalty}
\sigma|_F^{}=(\zeta_++\zeta_-)^{-2},
\end{equation}
 with $\zeta_-=0$ in the case of a boundary face.

\begin{definition}
	The \emph{robust interior penalty discontinuous Galerkin method (RIPDG)}  is the Galerkin procedure defined by \eqref{wip} with $\mbf{w}=(w_+,w_-)$ for $w_*$ given by \eqref{new_weights}, $*\in\{+,-\}$ and the discontinuity-penalization parameter selected as in \eqref{robust_penalty}.
\end{definition}

\begin{remark}\Rev{The elementary identity $(\zeta_++\zeta_-)^{-2}\le(\min_{*\in\{+,-\}} 2 \zeta_*)^{-2}$, results into the bound 
		\begin{equation}\label{relation Harmonic weight}
			\sigma|_F^{}\leq \min_{*\in\{+,-\}} m_{\k_*}C^2_{{\rm inv},*}\|a\mbf{n}|_{\k_*}\|^2_{L_\infty(F)}\|a^{-\frac{1}{2}}\|^2_{L_\infty(\k_*)}. 
		\end{equation}
		The estimate \eqref{relation Harmonic weight} shows that the RIPDG discontinuity-penalization parameter grows proportionally  $\min_{*\in\{+,-\}} C_{{\rm inv},*}$ and is effectively independent of the local size of the diffusion coefficient. In contrast, for IPDG, the discontinuity-penalization parameter grows proportionally  with $\max_{*\in\{+,-\}} C_{{\rm inv},*}$ and is not robust with respect to the contrast; cf., \eqref{tau}.
	}
\end{remark}

\begin{remark}
	The above argument still applies for different choices of element-wise polynomial Galerkin spaces by replacing suitably $C_{{\rm inv},*}$ by an available upper bound of the respective inverse estimate constant for the space at hand. 
\end{remark}

\begin{remark}[On inverse estimate constants]\label{rem_const}
For different choices of element-wise Galerkin spaces, we may end up with unspecified/hard to estimate universal constants in the inverse estimate. We note carefully that \emph{any such constants cancel out in the definition of the weights $\mbf{w}$ and, thus, the weights are independent from such constants}, ensuring the practical selection of weights in various approximation space scenarios. Of course, as is the case for the standard IPDG method, the RIPDG penalty parameter will depend proportionally on such universal constants.
\end{remark}

The proof of \eqref{cont} is also then immediate upon considering the inconsistent formulation 
\[
\begin{aligned}
	B(\Rev{z},v):=&\ \int_\Omega a\bnabla \Rev{z}\cdot \bnabla v\ud x +\int_\Gamma \sigma \jump{\Rev{z}}\cdot \jump{v}\ud s \\
	&-\int_\Gamma  \big({\wmean{a\Pi_{\bf{p-1}} \nabla \Rev{z}}}\cdot \jump{v}+ {\wmean{a\Pi_{\bf{p-1}} \nabla v}}\cdot \jump{\Rev{z}}\big)\ud s,
\end{aligned}
\]	
for all $\Rev{z},v\in\mathcal{V}$,  resulting in the continuity bound
\begin{equation}\label{cont_two}
	B(w,v)\le \frac{3}{2} \ndg{\Rev{z}}\ndg{v},\qquad \Rev{z},v\in\mathcal{V}.
\end{equation}

We now compare the effects on the size of $\sigma$ of the classical choice $\mbf{w}=(1/2,1/2)$ and the of the one  proposed herein for the averaging operator. 
Any significantly different behaviour is bound to occur in extreme mesh and/or polynomial degree local variation scenarios. To that end, we consider the problem \eqref{pde} for $d=2$ with $a=I_{2\times 2}$, the identity matrix. The latter's solution is approximated via \eqref{wip} for $\mathbf{w}=(1/2,1/2)$ and also via \eqref{wip} with the choice \eqref{new_weights}, over a mesh consisting of two quadrilateral elements
\[
K_1=(0,1-\delta)\times (0,1),\qquad K_2=(1-\delta,1)\times (0,1),
\]
for $\delta <1/2$ and with uniform local polynomial degree $p$. The \Rev{discontinuity-penalization} parameter will differ for the two variants \Rev{only on the interior face} $F=\{1-\delta\}\times (0,1)$. For the classical IPDG method (\eqref{wip} with $\mathbf{w}=(1/2,1/2)$), we calculate:
\[
 \sigma_F\equiv\sigma_F^{\rm IP}=4p(p+1)/\delta,
\]  
whereas for the choice \eqref{new_weights}, we get
 \[
 \sigma_F\equiv\sigma_F^{\rm RIP}=\frac{8p(p+1)}{(\sqrt{1-\delta}+\sqrt{\delta})^2} .
 \]  
So, as $\delta\to 0$, we have $\sigma_F^{\rm IP}\to\infty$, while 
$
\sigma_F^{\rm RIP}\to 8 p(p+1). 
$
This extreme scenario highlights vastly different penalization pattern between the two variants. In the new choice of the weighted version of IPDG (\eqref{wip} with  \eqref{new_weights}), we observe \emph{robust} behaviour with respect to the local mesh variation. 

Completely analogously, we now set $\delta=1/2$, so that $K_1$ and $K_2$ are identical in terms of shape, and we set $p_{K_1}^{}=1$ and $p_{K_2}^{}=p>1$. We then compute
 \[
\sigma_F^{\rm IP}=8p(p+1)\to \infty,
\]  
as $p\to\infty$,
whereas
\[
\sigma_F^{\rm RIP}= 16\big( (p(p+1))^{-1/2}+2^{-1/2}\big)^{-2}\to 32.
\]

Having the penalty parameter converging to infinity may become an issue in terms of the conditioning of the resulting linear systems. Therefore,  when a conforming subspace of the DG space with the same approximation properties exists, e.g., for the case of a simplicial mesh, one expects only conditioning differences between the two variants. However, in the more practically interesting case whereby there is \emph{no} conforming subspace of full order, e.g., for the case of $\fes$ on a mesh $\mesh$ consisting of box-type (quadrliateral/hexahedral) elements, $\sigma_F^{\rm IP}\to\infty$ may be compromising in terms of approximation also.

%

\section{Extension to general polytopic elements}\label{poly_RIPDG}
The classical IPDG has been put forward as a method of choice for the development of Galerkin methods on meshes consisting of extremely general polygonal/polyhedral elements; see, e.g., \cite{cangiani2013hp,book,dg-ease,BassiJCP2012} and the references therein. A key development for practical applications has been their ability to admit provably stable approximations on elements $\k$ with \emph{arbitrary} number of faces $F$, possibly with $|F|\ll |\k|^{1/d}$. The weight selection in the robust IPDG method presented above depends on the number of elemental faces. Applying the method in the form \eqref{wip} with the selection \eqref{new_weights} is may lead to spurious, excessive over-penalization as $m_{\k_*}\to\infty$.  

Fortunately, it is often possible to use \emph{non-overlapping} simplicial subdivisions of the element $\k$ to effectively \Rev{arrive} at a bounded $m_\k$. More specifically, a number of extremely mild geometric conditions are presented in \cite{dg-ease} (cf.~also \cite{book} for earlier ideas in this vein) so that a polygonal/polyhedral element $\k$, possibly containing curved faces, admits an inverse estimate of the form
\begin{equation}\label{inv_dg-ease}
	\ltwo{v}{L_2(F_i)}^2\le C^2_{{\rm inv}}(p,F_i,\k)\ltwo{v}{L_2(\k)}^2,
\end{equation} 
for any $v \in\mathcal{P}_p(\k)$ and for $\{F_i\}_{i=1}^{m_\k}$ a mutually disjoint subdivision of $\partial\k$, with $C_{{\rm inv}}(p,F_i,\k)>0$ a concretely determined expression; see \cite[Lemma 4.21]{dg-ease} for the exact formula. A key idea in \cite{dg-ease} is that each $F_i$ is star-shaped with respect to one interior point $x_i$ of $\k$, $i=1,\dots,m_\k$, \emph{and} may contain an arbitrary number of (possibly very small in size) $(d-1)$-dimensional faces. 
Following the same line of argument as in Section \ref{RIPDG}, we can define a robust IPDG method on meshes consisting of general, possibly curved, polygonal/polyhedral elements by selecting
	\[
\zeta_*:=\Big(2\sqrt{m_{\k_*}^{}}C_{{\rm inv}}(p,F_i,\k)\|a\mbf{n}|_{\k_*}\|_{L_\infty(F)}\|a^{-\frac{1}{2}}\|_{L_\infty(\k_*)}\Big)^{-1}, 
\]
 and retaining the definitions of \eqref{new_weights} and \eqref{robust_penalty}. We note carefully that in the case of general polytopic meshes \eqref{inv_dg-ease} may contain an unknown constant which, nevertheless, cancels out in the definition of the weights as discussed in Remark \ref{rem_const}, yielding explicit choice for the weights $\mbf{w}$.

\Rev{
\section{On box-type elements}\label{sec:box}
In Section \ref{RIPDG}, the discussion focused on simplicial meshes, for which weights and discontinuity-penalization parameters for RIPDG, that are free of unknown constants are provided. The main technical tool has been the availability of a trace inverse estimate with \emph{explicitly known} constants for the simplicial reference element \cite{warburton2003constants}, also presented in Lemma \ref{improved_inv} for an arbitrary simplex. }

\Rev{
For box-type (quadrilateral, hexahedral, etc.) element meshes, classical discontinuous Galerkin literature and popular available implementations in software libraries propose the use of `$Q$-type' element bases (i.e., polynomial spaces of degree $p$ in \emph{each} variable) in conjunction with nonlinear element mappings. In other words, we are in the setting of $\fesmapped$ defined above, in complete analogy with the classical developments in conforming finite element methods. In this case, the trace inverse estimate becomes
	\begin{equation}\label{inv_ineq_gen_Q}
		\ltwo{v}{L_2(F)}^2\le (p_\k^{}+1)^2\|J_{\mbf{F}_\k}|_{F\circ\mbf{F}_\k}\|_{L_\infty(F\circ\mbf{F}_\k)}^{}\|J_{\mbf{F}_\k}^{-1}\|_{L_\infty(K)}^{}\ltwo{v}{L_2(\k)}^2,
	\end{equation}
with $J_{\mbf{F}_\k}$ denoting the Jacobian of the mapping $\mbf{F}_\k:(0,1)^d\to K$, for a box-type element $K\in\mesh$.
The proof follows by employing Fubini's Theorem in conjunction with Lemma \ref{improved_inv} for $d=1$ to show the trace inverse estimate
\[
	\ltwo{v\circ\mbf{F}_\k}{L_2(F\circ\mbf{F}_\k)}^2\le (p_\k^{}+1)^2\ltwo{v\circ\mbf{F}_\k}{L_2((0,1)^d)}^2,
\]
for any $(d-1)$-dimensional  interface $F\circ\mbf{F}_\k$ of the reference hypercube $(0,1)^d$, along with a standard scaling argument. This setting was alluded to in Remark \ref{rem_const} above. Therefore, employing \eqref{inv_ineq_gen_Q} to construct the weights from RIPDG will involve the Jacobian of the element maps; this may be somewhat inconvenient in certain practical scenarios.}

\Rev{An alternative point of view was proposed in \cite{cangiani2013hp} (see also \cite{cangiani2015hp,book,cangiani2017hp,DGpolyBiharmonic,dong_exponent,dg-ease} for further results), whereby the finite element space is defined via  $\fes$ for box-type elements also, that is element-wise polynomials of \emph{total} degree $p$ are employed even for box-type elements, without the use of non-linear mappings. As it was shown numerically in \cite{cangiani2013hp,cangiani2017hp} and proven in \cite{dong_exponent}, this choice results into improved spectral convergence under $p$-refinement in various settings, since less basis functions per element are used. In this spirit, we may view a box-type element $K\in\mesh$ as the union of $d$ non-overlapping simplicial elements and, correspondingly each $(d-1)$-dimensional box-type interface can be viewed as a union of  more than one $(d-1)$-dimensional simplicial faces. In such a setting, we can employ the developments of Section \ref{poly_RIPDG} to define the RIPDG weights and discontinuity-penalization parameter, exactly like in the case of a general polytopic element. The advantage of this approach is that nonlinear mappings are not involved in the constants, at the expense of a ``notional subdivision'' of the box-type element into non-overlapping, simplicial subelements.. Note that this process is not increasing the number of degrees of freedom: it is performed locally only for the definition of the weights and of $\sigma$.
}

\section{Robust IPDG for degenerate elliptic problems}\label{RIPDG-deg}
It is possible to have well-posed elliptic problems of the form \eqref{pde} with diffusion tensors $a$ for which $a\mbf{n}=\mbf{0}$, or even $a=\mbf{0}$, on lower-dimensional manifolds of $\Omega$ \cite{or73}. In such cases, the choice of weights from \eqref{new_weights} has to be revisited. To that end, we consider the  \emph{inconsistent robust interior penalty discontinuous Galerkin method} reading: 
find $u_h\in\fes$, (or in $\fesmapped$,) such that
\begin{equation}\label{wip_deg}
	\tilde{B}(u_h,v_h)=\tilde{\ell}(v_h),\qquad\text{for all } v_h\in \fes,
\end{equation}
with 
\[
\begin{aligned}
	\tilde{B}(u_h,v_h):=&\ \int_\Omega a\bnabla u_h\cdot \bnabla v_h\ud x +\int_\Gamma \sigma \jump{ u_h}\cdot \jump{v_h}\ud s \\
	&-\int_\Gamma  \big(\wmean{\sqrt{a}\Pi_{\mbf{p-1}}(\sqrt{a}\nabla u_h)}\cdot \jump{v_h}+\theta \wmean{\sqrt{a}\Pi_{\mbf{p-1}}(\sqrt{a}\nabla v_h)}\cdot \jump{u_h}\big)\ud s,
\end{aligned}
\]
and
\[
\begin{aligned}
	\ell(v_h):=&\ \int_\Omega f v_h\ud x +\int_{\partial\Omega} g\big(\sigma v_h-\theta \sqrt{a}\Pi_{\mbf{p-1}}(\sqrt{a}\nabla v_h\cdot\mbf{n})\big) \ud s ,
\end{aligned}
\]
for $\theta\in[-1,1]$, with $\Pi:L_2(\Omega)\to\fes$ (or, $\Pi:L_2(\Omega)\to\fesmapped$,) denoting the orthogonal $L_2$-projection operator onto the dG space. The method was first presented in \cite{georgoulis2006note} for the case $\mbf{w}=(1/2,1/2)$ and with $\Pi_{\mbf{p-1}}$ replaced by $\Pi_{\mbf{p}}$. For simplicity of the presentation, we shall again focus on the symmetric case $\theta=1$; the cases $\theta\in[-1,1)$ follow via trivial modifications of the arguments given below.

From \eqref{inv_ineq_gen}, and working as before, we get, for $F\subset \partial \k_+\cap
\partial \k_-$,
\[
\begin{aligned}
	&	\Big|2\int_\Gamma  \wmean{\sqrt{a}\Pi_{\mbf{p-1}}(\sqrt{a}\nabla u_h)}\cdot \jump{v_h}\ud s \Big|\\
	\le& \ 2
	\sum_{F\subset \Gamma}\int_F \Big(\sum_{*\in\{+,-\}} w_*\|\sqrt{a}\mbf{n}|_{\k_*}\|_{L_\infty(F)}|\Pi_{\mbf{p-1}}(\sqrt{a}\nabla u_h)|_{\k_*}| \Big)| \jump{v_h}|\ud s\\
	\le &\ 2 \sum_{F\subset \Gamma}\sum_{*\in\{+,-\}} \!\!\!C_{{\rm inv},*}w_*\|\sqrt{a}\mbf{n}|_{\k_*}\|_{L_\infty(F)}\ltwo{\sqrt{a}\nabla u_h}{L_2(\k_*)} \ltwo{\jump{v_h}}{L_2(F)},
\end{aligned}
\]
using the stability of the orthogonal $L_2$-projection operator.
Setting now 
\[
\tilde{\zeta}_*:=\Big(2\sqrt{m_{\k_*}}C_{{\rm inv},*}\|\sqrt{a}\mbf{n}|_{\k_*}\|_{L_\infty(F)}\Big)^{-1}, 
\]
for $*\in\{+,-\}$, and selecting, again, 
\begin{equation}\label{new_weights_deg}
	w_* =\frac{\tilde{\zeta}_*}{\tilde{\zeta}_++\tilde{\zeta}_-},
\end{equation}
we deduce
\[
\begin{aligned}
	\Big|2\int_\Gamma  \wmean{\sqrt{a}\Pi_{\mbf{p-1}}\sqrt{a}(\nabla u_h)}\cdot \jump{v_h}\ud s \Big|	\le &\ \frac{1}{2} \ltwo{\sqrt{a}\bnabla u_h}{L_2(\Omega)}^2
	+\frac{1}{2}\sum_{F\subset \Gamma}\sigma|_F\ltwo{\jump{v_h}}{L_2(F)}^2,
\end{aligned}
\]
with 
$\sigma|_F:=(\tilde{\zeta}_++\tilde{\zeta} _-)^{-2},
$ 
for every face $F\subset\Gamma$. If the limit $\sqrt{a}\mbf{n}|_{\k_*}\to \mbf{0}$ (from within the element interior to the face) is valid a.e.~on $F$ for exactly one of $*\in\{+,-\}$, we have $w_*\to1$, with the other weight in the pair tending to zero; also, in this case we have  $\sigma\to 0$.  
Further, if we have $\sqrt{a}\mbf{n}|_{\k_\pm}\to \mbf{0}$ from within the element interior to the face for both traces  a.e.~on $F$, we \Rev{adopt the convention that both $w_\pm\to0$ and $\sigma\to0$ on $F$ (i.e., we no longer enforce that $w_++w_-=1$)} .  

\section{A priori error analysis}\label{sec:apriori}
We now show that standard a priori error bounds hold for all the robust IPDG methods described in the previous section, assuming only sufficient regularity of the exact solution. In particular, \emph{no} `local bounded variation/local quasi-uniformity' mesh and polynomial degree assumptions are required, as is standard in the error analysis of the classical IPDG method with $\mbf{w}=(1/2,1/2)$.

More specifically, let $u\in H^{3/2+\epsilon}(\Omega)$, $\epsilon>0$. \Rev{Elementary calculations imply }
\begin{equation}\label{inconsistency form}
\Rev{	B(u_h,w_h) = B(u,w_h) + \int_\Gamma r(u) \cdot \jump{w_h}\ud s, \qquad w_h\in \fes,}
\end{equation}
with \Rev{$r(u):=\wmean {a\nabla u-\Pi_{\mbf{p-1}}(a\nabla u)}$} for the robust IPDG from Section \ref{RIPDG}, or $r(u):=\wmean {a\nabla u-\sqrt{a}\Pi_{\mbf{p-1}}(\sqrt{a}\nabla u)}$ for the variant from Section \ref{RIPDG-deg}.

Then, the coercivity, the continuity along with standard calculations imply
\begin{equation}\label{strang}
\ndg{u-u_h}\le {4}\inf_{v_h\in\fes}\ndg{u-v_h}+2\sup_{0\neq w_h\in \fes}\frac{\int_\Gamma r(u)\cdot \jump{w_h}\ud s}{\ndg{w_h}}.
\end{equation}

To estimate the inconsistency term, assuming further that $u\in H^2(\Omega)$, we make use of the best approximation estimate from \cite{chernov} (see also \cite{newpaper,thesis} for earlier results on box-type elements): there exists a constant $C_{\rm ap}>0$, independent of $v$ and of $p$ such that
\begin{equation}\label{best_app_l2_ref}
\Rev{	\norm{v-\Pi_{\mbf{p}}v}{L_2(\hat{F})}^2\le C_{\rm ap} (p+1)^{-1}\ltwo{\nabla v}{L_2(\hat{\k})}^2
},
\end{equation}
for any face $\hat{F}\subset \partial\hat{\k}$ of a reference simplicial element $\hat{\k}$. Employing a standard scaling argument, \eqref{best_app_l2_ref} implies for an element $\k\in\mesh$ the following best approximation estimate:
\begin{equation}\label{best_app_l2}
	\norm{v-\Pi_{\mbf{p}}v}{L_2(F)}^2\le C_{\rm ap} \frac{|F|h_\k^2}{|\k|(p+1)}\ltwo{\nabla v}{L_2(\k)}^2.
\end{equation}
Upon observing the identity $w_*\sigma^{-1/2}=\zeta_*$ (or $w_*\sigma^{-1/2}=\tilde{\zeta}_*$ for the method of Section \ref{RIPDG-deg}), we have, respectively, for any $\mbf{v}_h\in [\fesmone]^d$, 
\[
\begin{aligned}
&\int_\Gamma r(u)\cdot \jump{w_h}\ud s\\
\le & 
	\sum_{F\subset \Gamma}\!\sum_{*\in\{+,-\}}\!\!\! \zeta_*\|a\mbf{n}|_{\k_*}\|_{L_\infty(F)}\norm{\!\big(\nabla u -\mbf{v}_h-\Pi_{\mbf{p-1}}(\nabla u -\mbf{v}_h)\big)\!|_{\k_*}\!}{L_2(F)}  \norm{\sqrt{\sigma}\jump{w_h}}{L_2(F)}\\
\le & \sum_{F\subset \Gamma}\sum_{*\in\{+,-\}} \!\!\!\frac{\sqrt{C_{{\rm ap}}|F|}h_{\k_*}}{\sqrt{|\k_*|p_{\k_*}}}\zeta_*\|a\mbf{n}|_{\k_*}\|_{L_\infty(F)}|\nabla u-\mbf{v}_h|_{H^1(\k_*)} \norm{\sqrt{\sigma}\jump{w_h}}{L_2(F)}\\
\le & \sum_{F\subset \Gamma}\sum_{*\in\{+,-\}} \!\!\!	\frac{\|\sqrt{a}\|_{L_\infty(\k_*)}\sqrt{dC_{{\rm ap}}}}{\Rev{2}\sqrt{m_{\k_*}}\sqrt{ p_{\k_*}^{}\!\!\!+d\!-\!1}}\frac{h_{\k_*}}{p_{\k_*}^{}}|\nabla u-\mbf{v}_h|_{H^1(\k_*)} \norm{\sqrt{\sigma}\jump{w_h}}{L_2(F)}
\\
\le & \ \Big(\frac{dC_{{\rm ap}}}{\Rev{2}}\sum_{\k \in \mesh}	\|\sqrt{a}\|_{L_\infty(\k)}^2\frac{h^2_{\k}}{p_{\k}^{3}}|\nabla u-\mbf{v}_h|_{H^1(\k)}^2\Big)^{\frac{1}{2}} \norm{\sqrt{\sigma}\jump{w_h}}{L_2(\Gamma)},
\end{aligned}
	\] 
 using \eqref{best_app_l2} in the first step and the definition of $\zeta_*$ in the penultimate step. The last estimate concerns the robust IPDG method from Section \ref{RIPDG} and holds under the regularity assumption $u\in H^2(\Omega)$.
	
		Correspondingly, for the method of Section \ref{RIPDG-deg}, assuming that the $a$ is such that $\sqrt{a}\nabla u \in [H^1(\Omega)]^d$, (\emph{instead of} $u\in H^2(\Omega)$ as before,)  we arrive at
\[
\begin{aligned}
	\int_\Gamma r(u)\cdot \jump{w_h}\ud s
	\le & \ \Big(	\frac{dC_{{\rm ap}}}{\Rev{2}}\sum_{\k \in \mesh} \frac{h^2_{\k}}{p_{\k}^{3}}|\sqrt{a}\nabla u-\mbf{v}_h|_{H^1(\k)}^2\Big)^{\frac{1}{2}} \norm{\sqrt{\sigma}\jump{w_h}}{L_2(\Gamma)}.
\end{aligned}
\] 
Combining the above developments with \eqref{strang}, we deduce
\begin{equation}
	\begin{aligned}
\ndg{u-u_h}\le&\  {4}\inf_{v_h\in\fes}\ndg{u-v_h}\\
&\ +\sqrt{\Rev{2} d C_{\rm ap}}\inf_{\mbf{v}_h\in [\fesmone]^d}\Big(\sum_{\k \in \mesh}	\|\sqrt{a}\|_{L_\infty(\k)}^2\frac{h^2_{\k}}{p_{\k}^{3}}|\nabla u-\mbf{v}_h|_{H^1(\k)}^2\Big)^{\frac{1}{2}},
\end{aligned}
\end{equation}
for the robust IPDG of Section \ref{RIPDG}, or 
\[
		\ndg{u-u_h}\le  {4}\inf_{v_h\in\fes}\ndg{u-v_h}+\sqrt{\Rev{2}d C_{\rm ap}}\inf_{\mbf{v}_h\in [\fesmone]^d}\Big(\sum_{\k \in \mesh}	\frac{h^2_{\k}}{p_{\k}^{3}}|\sqrt{a}\nabla u-\mbf{v}_h|_{H^1(\k)}^2\Big)^{\frac{1}{2}},
\]
for the method of Section \ref{RIPDG-deg}, respectively. These basic estimates, together with standard $hp$-version best approximation results, give rise to standard a priori error bounds. Although not considered in detail here in the interest of brevity, following essentially the same steps as above, it is also possible to prove a priori error bounds for the robust IPDG method on polygonal/polyhedral elements from Section \ref{poly_RIPDG}, with the possible exception of considering simplicial coverings of the elements $\k\in\mesh$ along with extension operators; we refer to \cite{cangiani2013hp,book,dg-ease} for details.

The crucial development in the a priori error bounds above, compared to all respective bounds we are aware of in the literature, is that the best approximation estimates are \emph{completely localized element-wise}. This is in contrast to standard IPDG error bounds requiring local quasi-uniformity/bounded variation assumptions for the mesh size and local polynomial degree. This property is expected to be of importance on approximation space scenarios admitting extreme local variation.

\begin{remark}
In the above best approximation results we have taken a somewhat ``classical'' viewpoint of assuming sufficient regularity of the exact solution. More recent developments e.g.~\Rev{the so-called medius analysis} \cite{gudi}, or the use of recovery operators \cite{veeser-zanotti}, can provide quasi-optimality under minimal regularity assumptions on the exact solution. These approaches, however, are currently available only for the $h$-version IPDG under standard local quasi-uniformity mesh \Rev{and diffusion coefficient bounded local variation} assumptions for which the robust IPDG method presented in this work offers little or no advantage. Nonetheless, the extension of such ``optimal'' error analysis frameworks to the RIPDG setting is certainly of  interest.
\end{remark}


\section{Numerical Experiments}\label{sec:numerics}
We now present a series of numerical experiments showcasing the comparative performance between the classical IPDG and RIPDG. All numerical experiments below use the locally \emph{total degree $p$} approximation space per element $\fes$ on box-type or general polytopic element shapes, noting carefully that no conforming subspace of $\fes$  of the same degree is available.  This choice of approximation spaces has been advocated in \cite{cangiani2013hp,book,dg-ease} leading to optimal approximation yet reduced complexity for box-type elements.

Given that the effect of high contrast diffusion coefficient has been discussed extensively in \cite{dryja,burman_zunino,ern_weighted}, we confine our study to the effects of extreme mesh and polynomial degree variation. Nonetheless, we expect that the inclusion of high contrast diffusion coefficient in the spirit of \cite{burman_zunino,ern_weighted} in conjunction with extreme local variation scenarios will only augment the behaviours presented below. We stress that in \emph{all} comparisons presented below both methods have exactly the same numerical degrees of freedom as they are applied to the same approximation space; the two respective implementations differ \emph{only} on the choice of the weights $\bf{w}$ and, correspondingly, to the choice of the penalty parameter. \Rev{Therefore, any differences observed in performance are due to the different methods implemented in the otherwise identical program, and not on any other algorithmic factors (e.g., choice of quadrature or condition number estimation routines, etc.)}.
\subsection{Example 1: singularly perturbed reaction-diffusion problem}
	\begin{figure}
	\includegraphics[height=4.5cm,width=4.5cm]{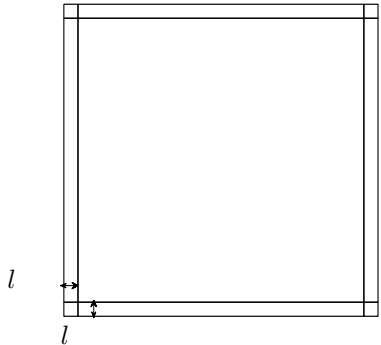}
	\put(-140,18){$l$}
	\put(-120,-3){$l$}
	\caption{Example 1. The $9$ element mesh.}\label{fig:ex1_mesh}
\end{figure}
We begin by testing the method on anisotropic rectangular elements in the context of layer adapted $hp$-version IPDG/RIPDG method.   Consider the problem
$$
-\epsilon \Delta u + u =f, \qquad \text{on } \Omega := (-1,1)^2,
$$	
with homogeneous Dirichlet boundary conditions and $f$ an analytic function chosen so that the exact solution reads
	\begin{equation}
		u(x,y):= \Big(1-\frac{\cosh(x/\sqrt{\epsilon})}{\cosh(1/\sqrt{\epsilon})} \Big)
		\Big(1-\frac{\cosh(y/\sqrt{\epsilon})}{\cosh(1/\sqrt{\epsilon})}\Big).
	\end{equation}
	This example has been studied extensively in \cite{melenk1999hp,thesis}. The solution exhibits boundary layers of thickness $\mathcal{O}(\sqrt{\epsilon})$. 
	
	To resolve the layers, we first use a layer-adapted anisotropic $9$-element mesh with characteristic width $l:=\min\{\lambda p \sqrt{\epsilon},0.5\}$, and some constant $\lambda>0$; see Figure \ref{fig:ex1_mesh} for an illustration.  It was shown in \cite{melenk1999hp} that conforming Galerkin methods on such meshes offer a robust exponential convergence under $p$ refinement.  Note that the maximum meshsize local variation constant $r$ for the $9$-element mesh is given by
	$
	r:= {2(1-l)}/{l}.
	$
	For $\epsilon \ll 1$,  we have
	$
	r\approx 2/ ( \lambda p\sqrt{\epsilon} ).
	$

	We compare the symmetric versions of the classical IPDG against the RIPDG method on the same meshes for polynomial degrees $p=1,\dots,7$. In Figure \ref{fig:ex1_error},  it is observed  that both methods converge exponentially in both dG-norm and broken $H^1$-seminorm $|\cdot|_{H^1(\Omega,\mesh)}:=\|\bnabla \cdot\|_{L_2(\Omega)}$ against the square root of the total numerical degrees of freedom ($DoFs$) under $p$ refinement for $\epsilon=10^{-5}$ and $\lambda=0.9$. The RIPDG appears to outperform IPDG in both measures. 
	\begin{figure}[t]
		\includegraphics[height=5.5cm,width=6cm]{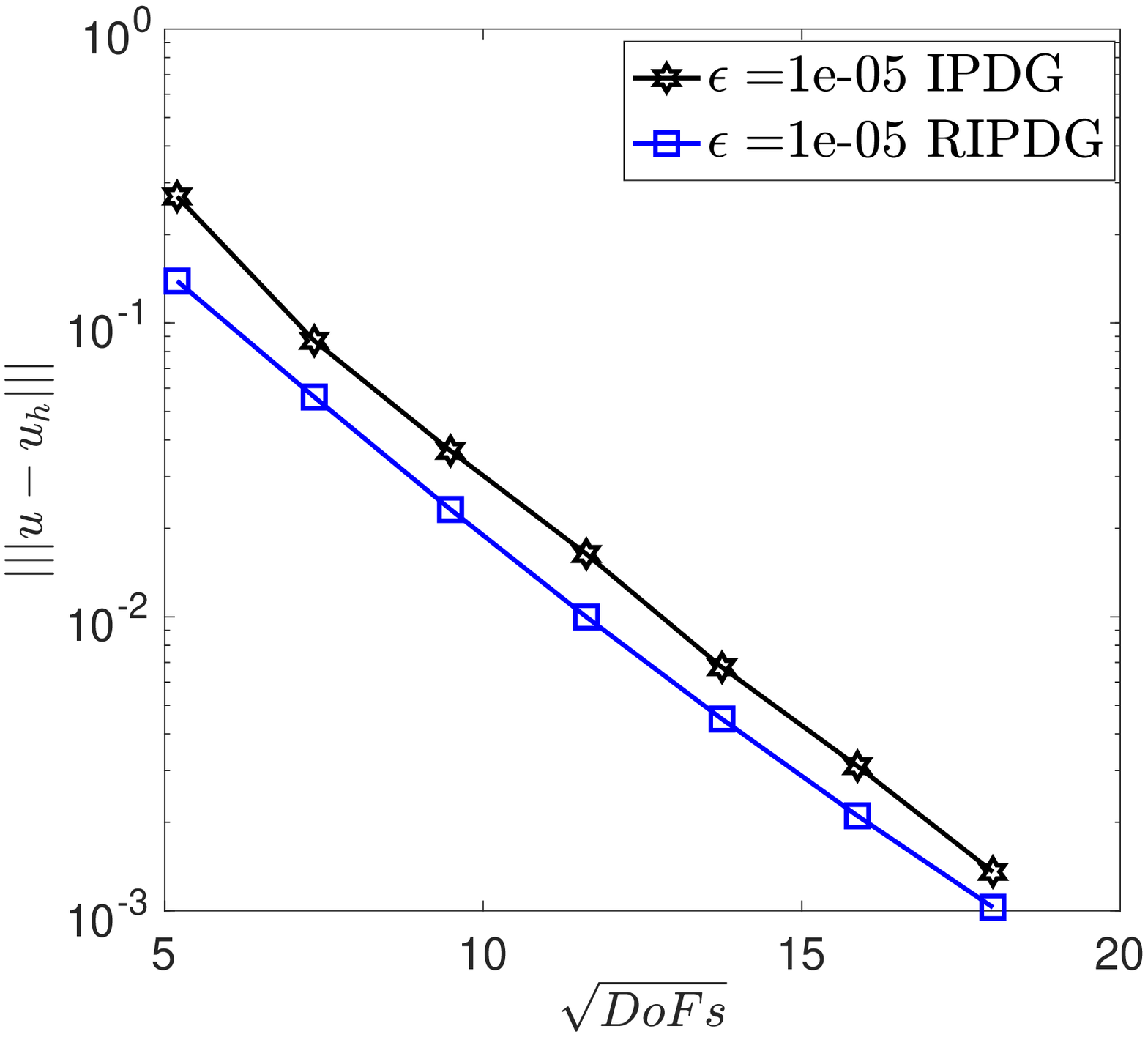}
		\includegraphics[height=5.5cm,width=6cm]{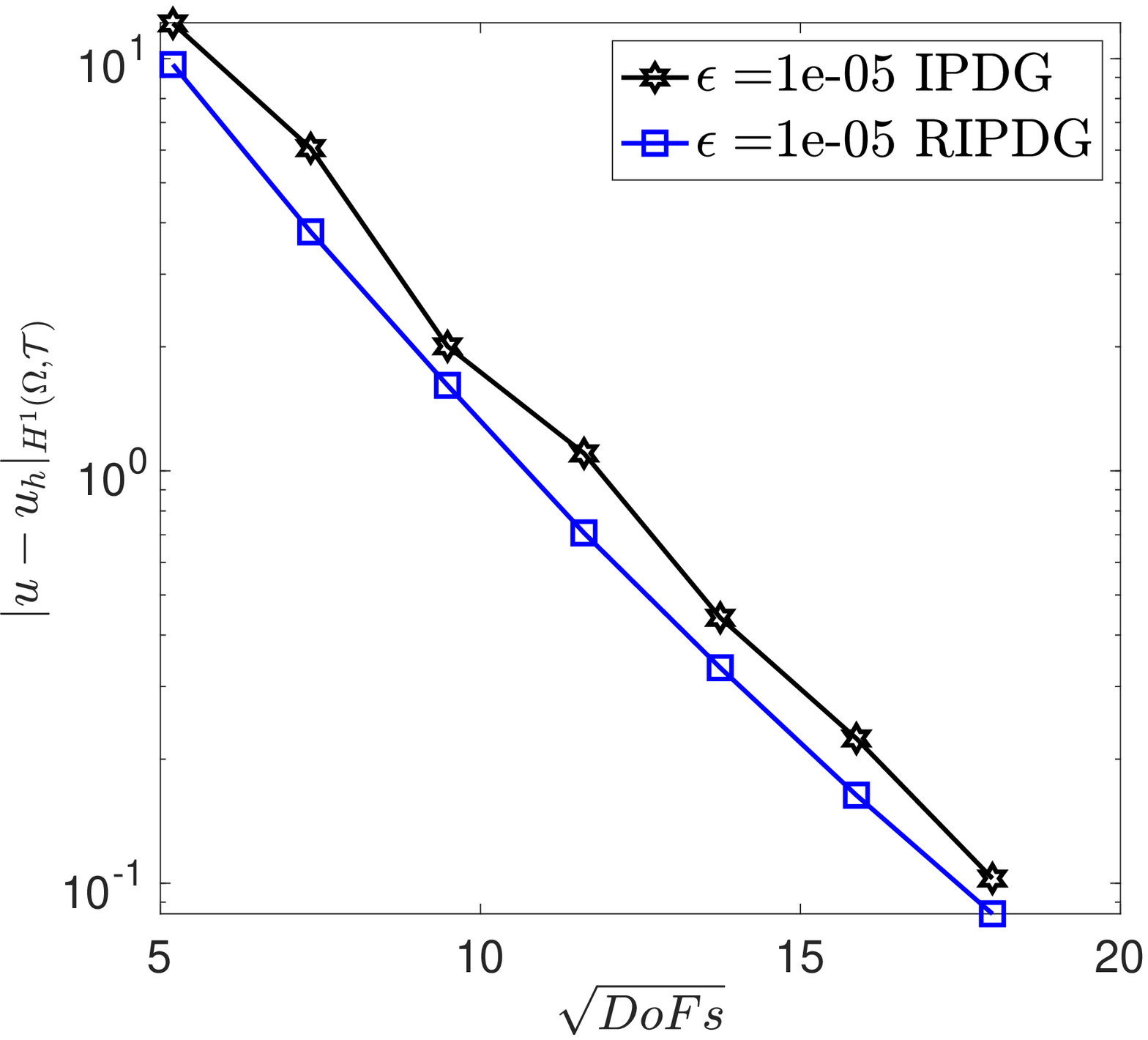}
		\caption{Example 1. Convergence in dG-norm (left)  and broken $H^1$-seminorm (right)  for $\epsilon=10^{-5}$ and $p=1,\dots,7$.}\label{fig:ex1_error}
	\end{figure}
\begin{figure}[h!]
	\includegraphics[height=5.5cm,width=6cm]{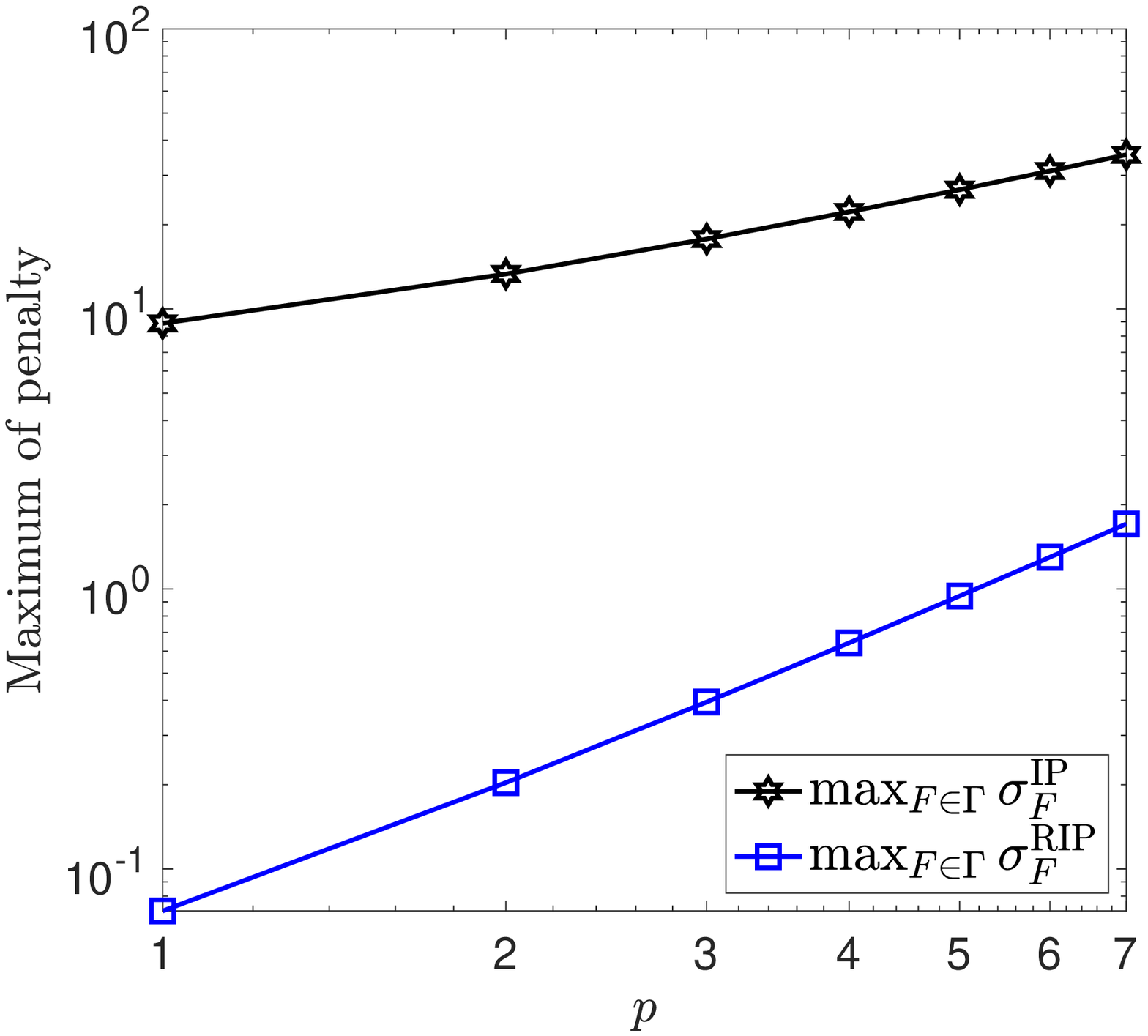}
	\includegraphics[height=5.5cm,width=6cm]{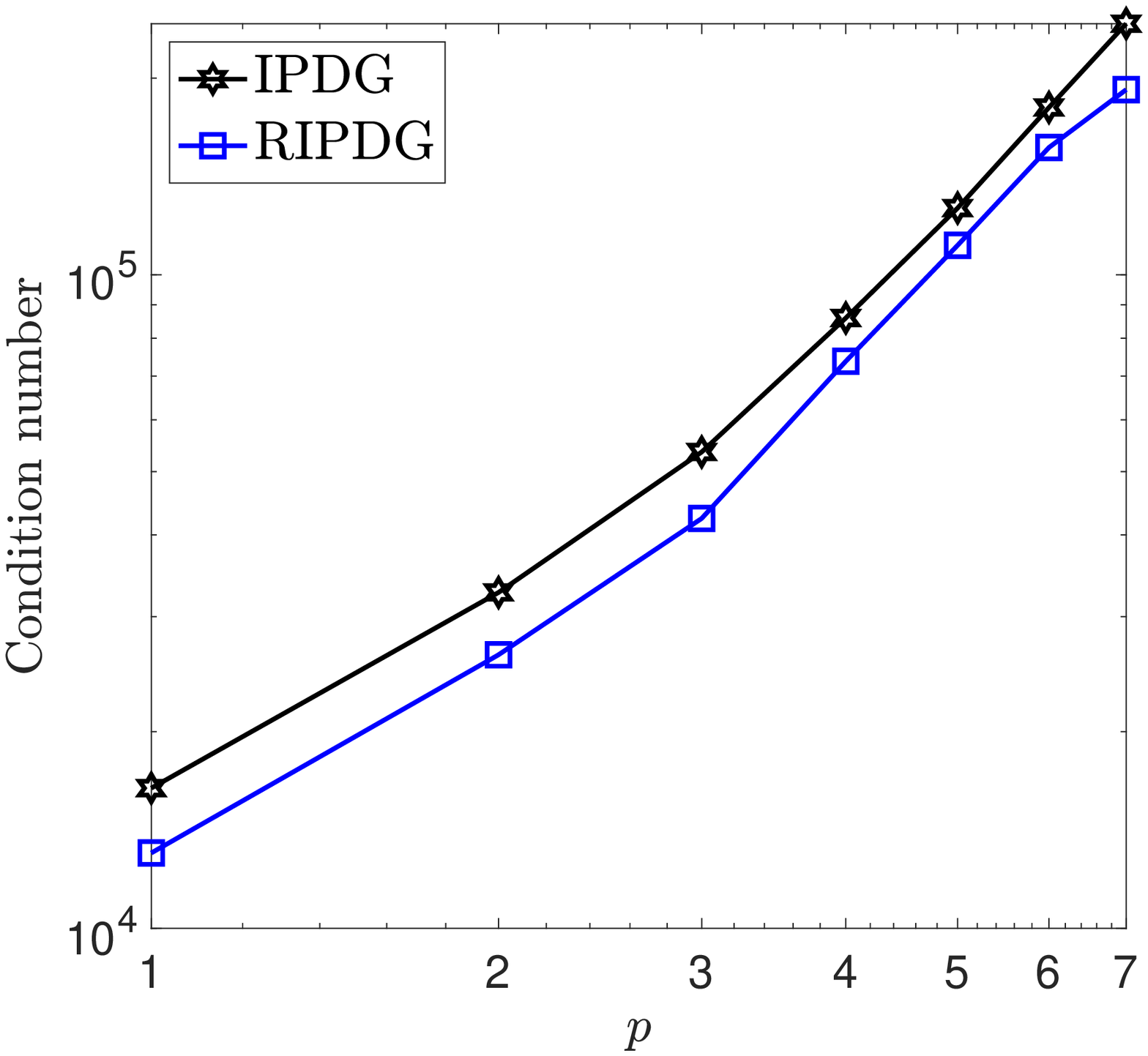}
	\caption{Example 1. Maximum of the penalty parameter (left) and the condition number of the linear system (right)  for $\epsilon=10^{-5}$ and $p=1,\dots,7$.}\label{fig:ex1_condition_no_sigma}
\end{figure}

In Figure \ref{fig:ex1_condition_no_sigma} (left), we compare the magnitudes of the global maximum of the penalty parameter values $\max_{F\subset\Gamma}\sigma_F^{\rm IP}$ and $\max_{F\subset\Gamma}\sigma_F^{\rm RIP}$, respectively, as well as the condition numbers for each method based on the same choice of basis. As expected, RIPDG requires a far smaller penalty parameter in theory for stability compared to the standard IPDG. For $p=1$, $\max_{F\subset\Gamma}\sigma_F^{\rm IP}$ is \textcolor{black}{$120$ times larger than} $\max_{F\subset\Gamma}\sigma_F^{\rm RIP}$. This difference in the definition of the penalty parameter explains also the improvement in the condition number for RIPDG compared to IPDG by a factor \textcolor{black}{ about 1.3}  for $p=1,\cdots,7$; see Figure \ref{fig:ex1_condition_no_sigma} (right).

In the above experiment, we observe only a modest improvement on error and the conditioning by using RIPDG as opposed to IPDG. We highlight that the $hp$-version spaces used for convergence admit a conforming subspace of the same approximation capabilities.

\begin{figure}[t]
	\includegraphics[height=5.5cm,width=6.2cm]{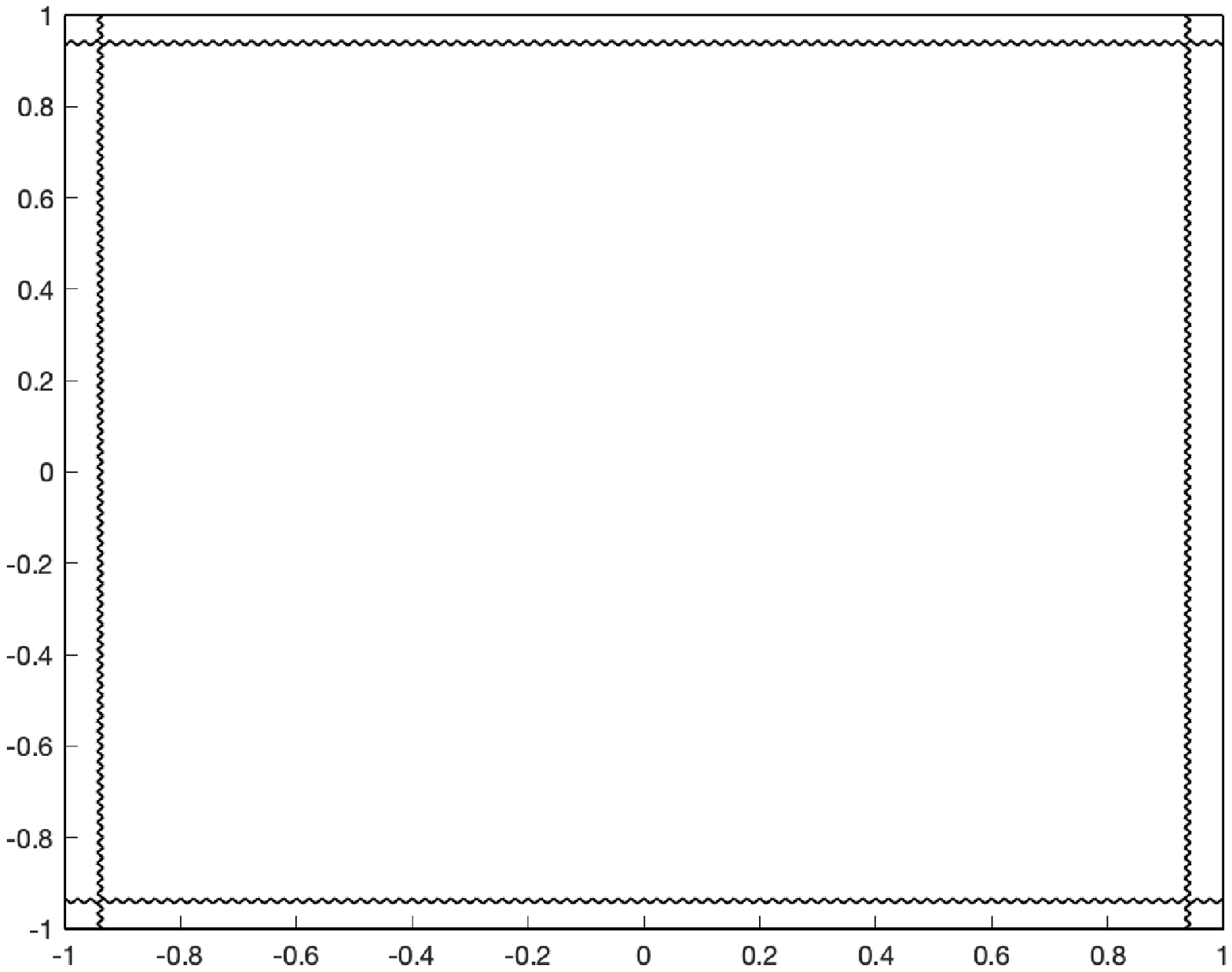}\hspace{-.5cm}
	\includegraphics[height=5.5cm,width=6.2cm]{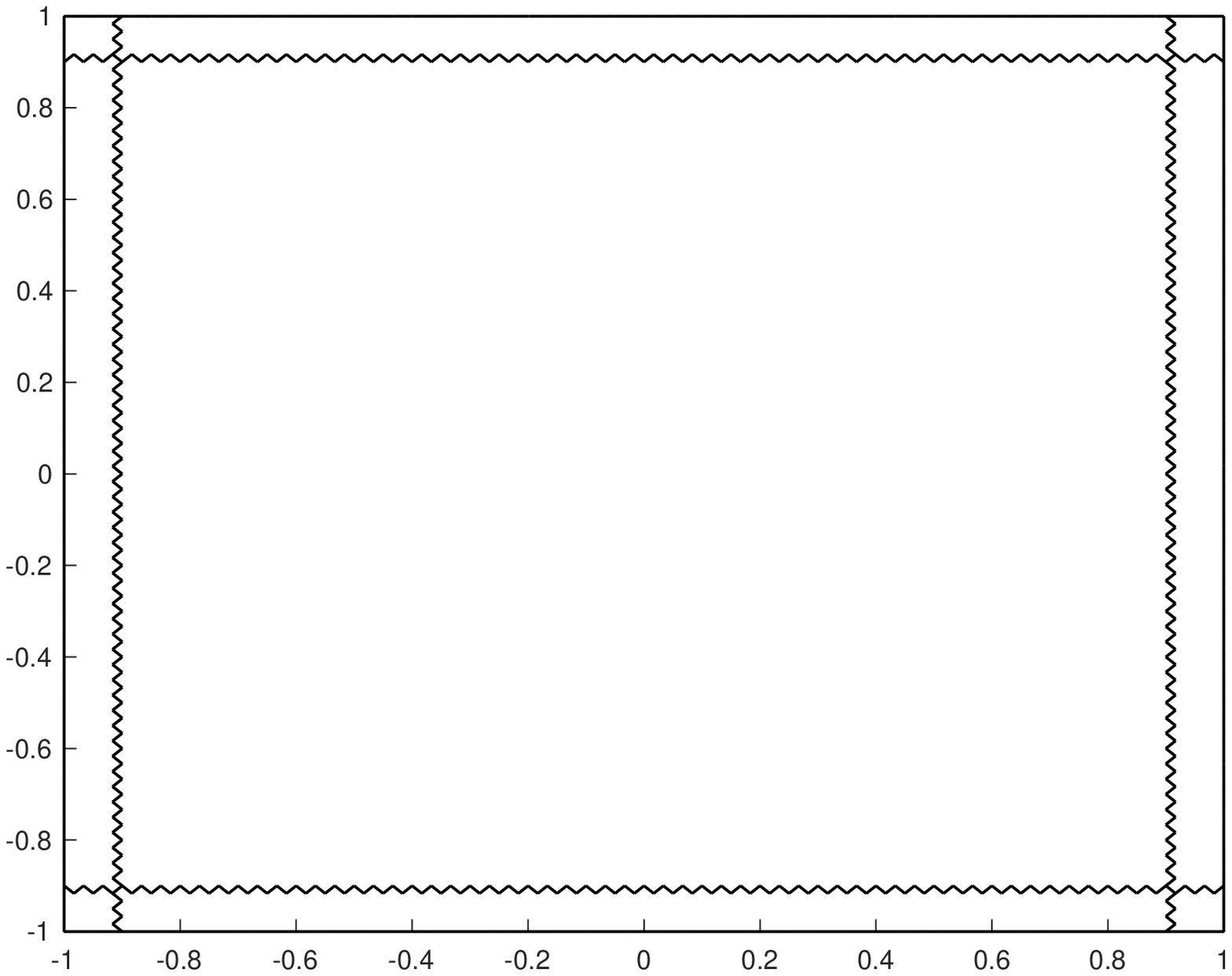}
	\includegraphics[height=5.5cm,width=6.2cm]{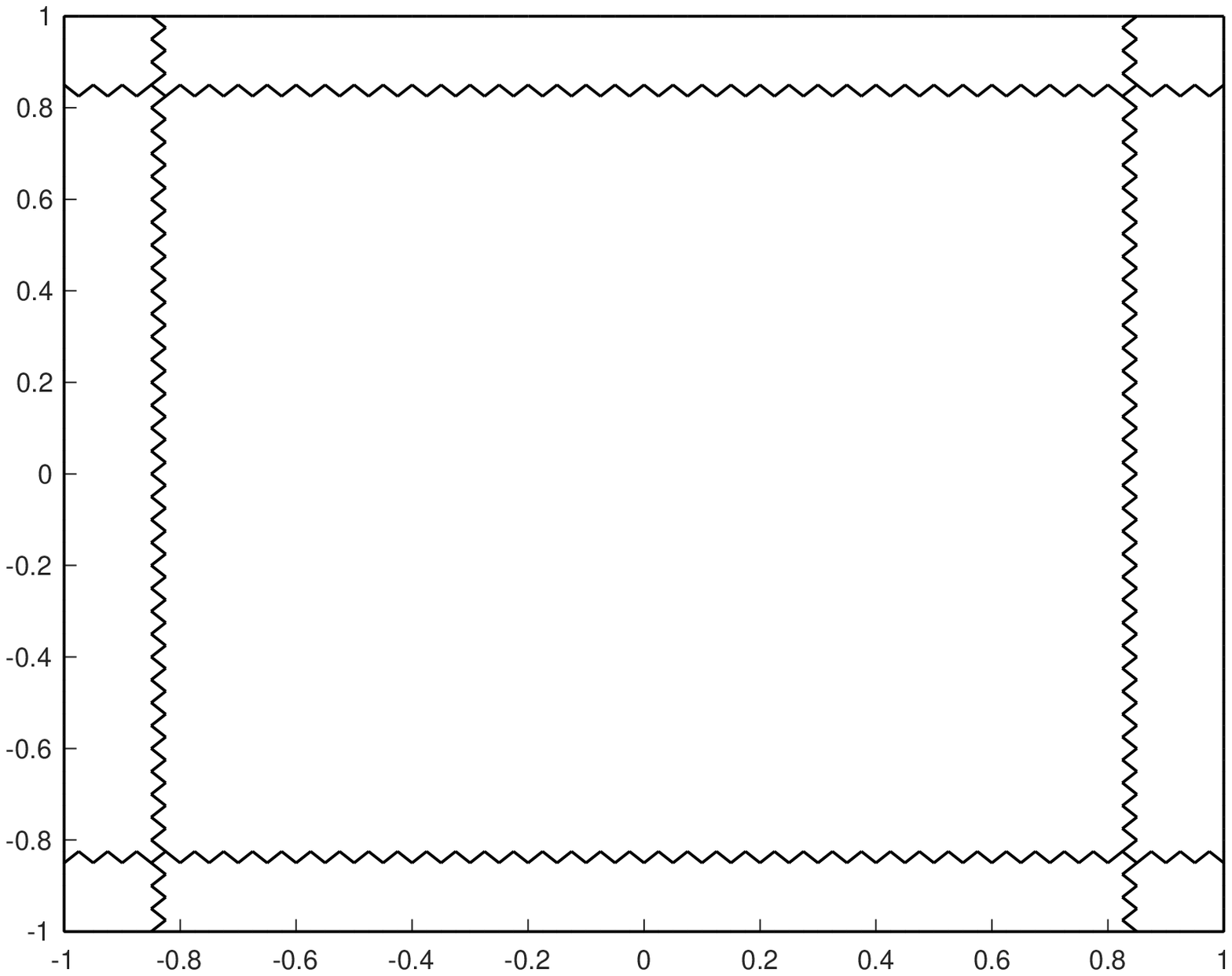}\hspace{-.5cm}
	\includegraphics[height=5.5cm,width=6.2cm]{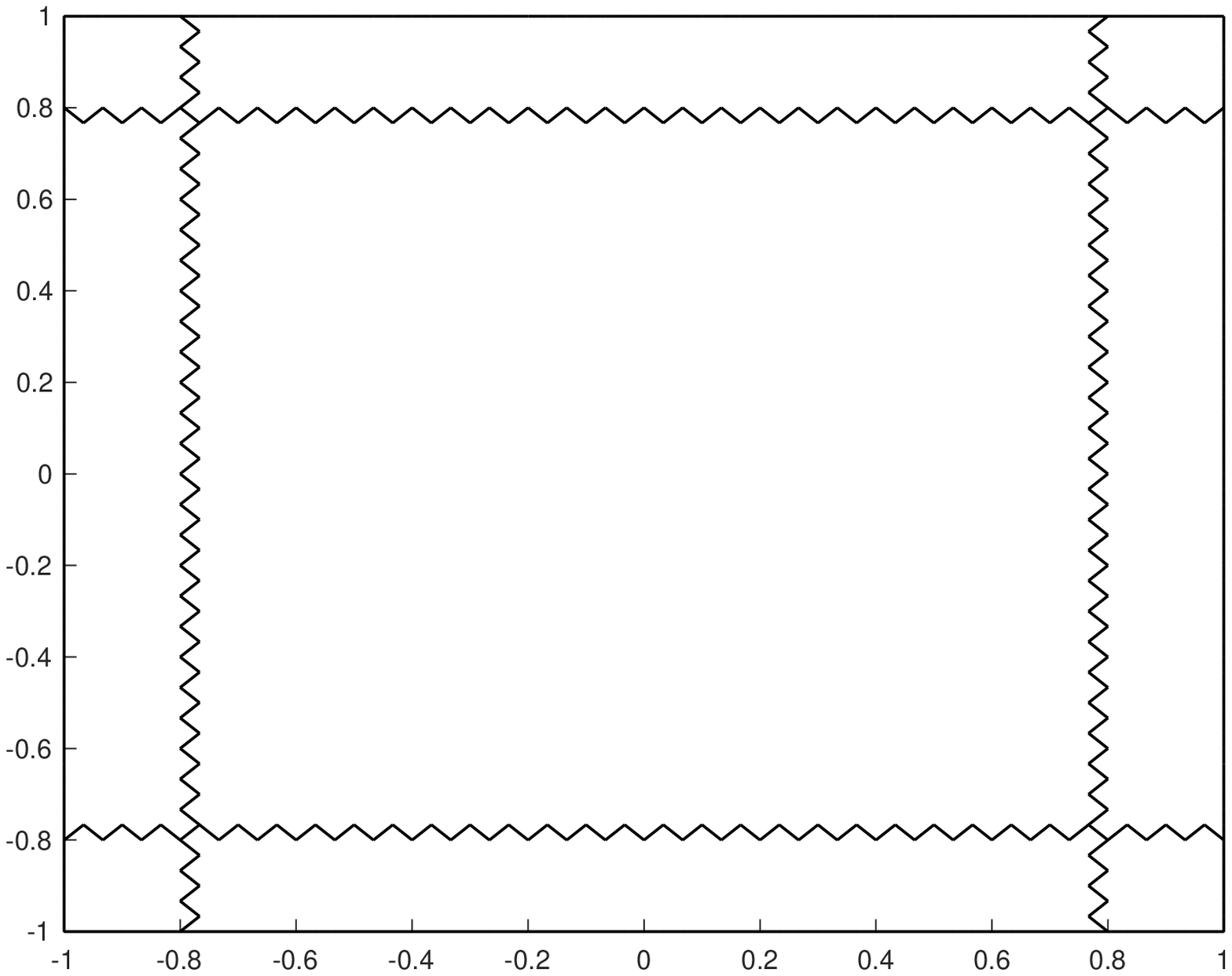}
	\caption{Example 1. ``Zigzag'' meshes for $p=2$ (top left), $p=3$ (top right), $p=5$ (bottom left), and $p=7$ (bottom right).  }\label{fig:ex1_zigzag_meshes}
\end{figure}

We now employ a ``zigzag'' polygonal version of the $9$ element mesh shown in Figure \ref{fig:ex1_zigzag_meshes} for $p=2,3,5,7$. In particular, we replace the interior faces from \Rev{the Figure} \ref{fig:ex1_mesh} by ``zigzag'' curves of amplitude {$l/6$} with respect to the length scale $l=\min\{\lambda p\sqrt{\epsilon},0.5\}$; the latter still characterises the distance of each interface to the non-intersecting portion of the boundary. On this $9$ element polytopic mesh, we solve the same problem using the RIPDG and IPDG methods, on physical (i.e., unmapped) element-wise polynomial spaces of uniform degree $p$ for ${p=2},\dots, 8$ and $\epsilon=10^{-3}$, with the penalty choice informed by the inverse estimate \eqref{inv_dg-ease}.  We observe that in this mesh \emph{no} conforming subspace of sufficient approximation capabilities is available. 

In Figure \ref{fig:ex1_condition_no_sigma_zigzag}, we compare the magnitudes of the global maximum of the penalty parameter values and the condition numbers for each method based on the same choice of basis. 	In Figure \ref{fig:ex1_error_zigzag}, we compare the respect errors for $p=1,\dots,8$. In Figure \ref{fig:ex1_error_zigzag},  it is observed  that both methods converge exponentially in  dG-norm, and broken $H^1$-seminorm against the square root of the total numerical degrees of freedom ($DoFs$) under $p$ refinement for $\epsilon=10^{-3}$ and $\lambda=0.9$. The RIPDG appears to significantly outperform IPDG for higher $p$; the IPDG appears to stagnate for $p=7,8$. We believe that the reason behind this behaviour is the IPDG's penalization magnitude in conjunction with the essential non-conformity of the approximation space. On the other hand, RIPDG's $p$-convergence appears not to be affected.

\begin{figure}[h!]
	\includegraphics[height=5.5cm,width=6cm]{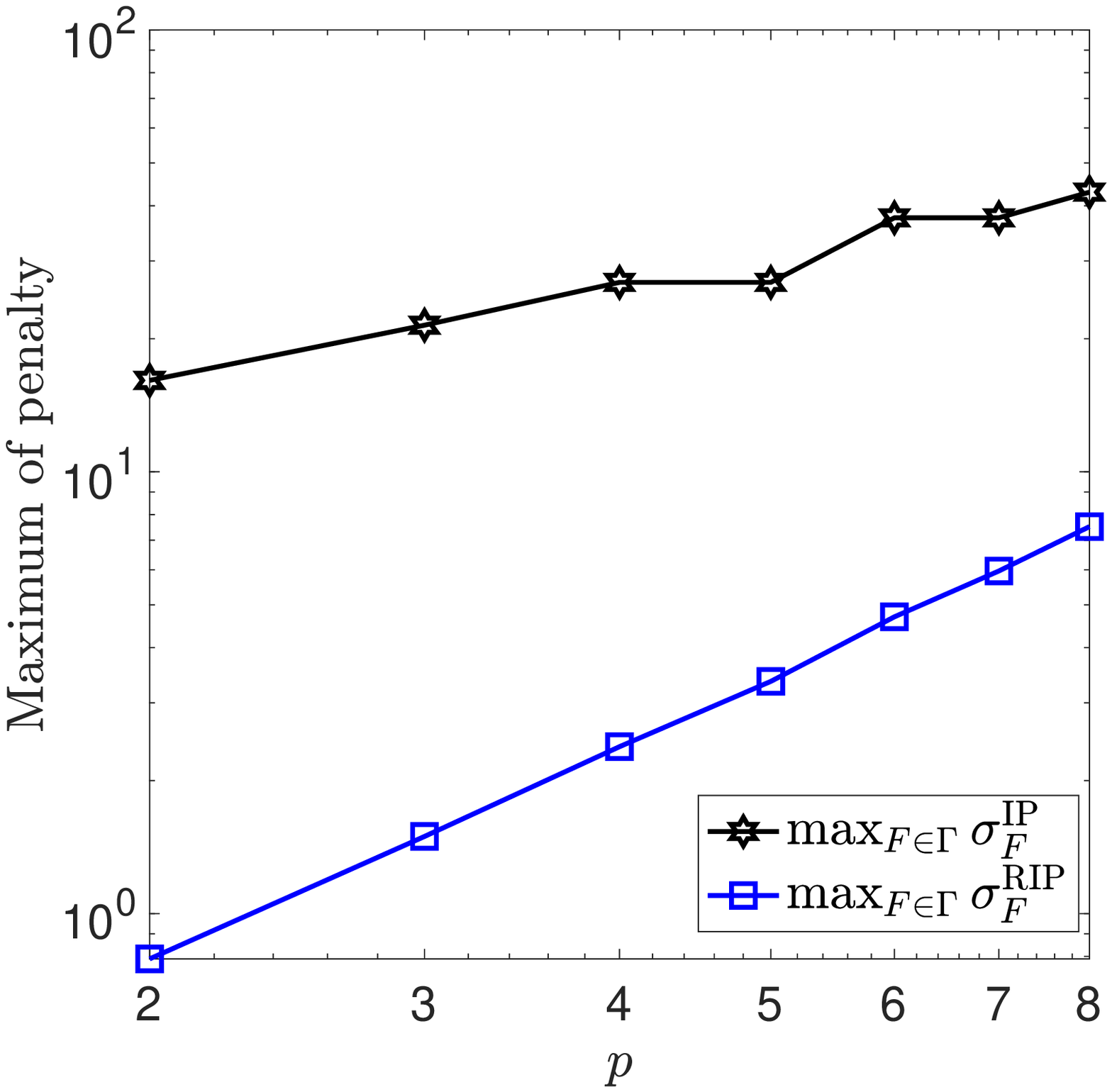}
	\includegraphics[height=5.5cm,width=6cm]{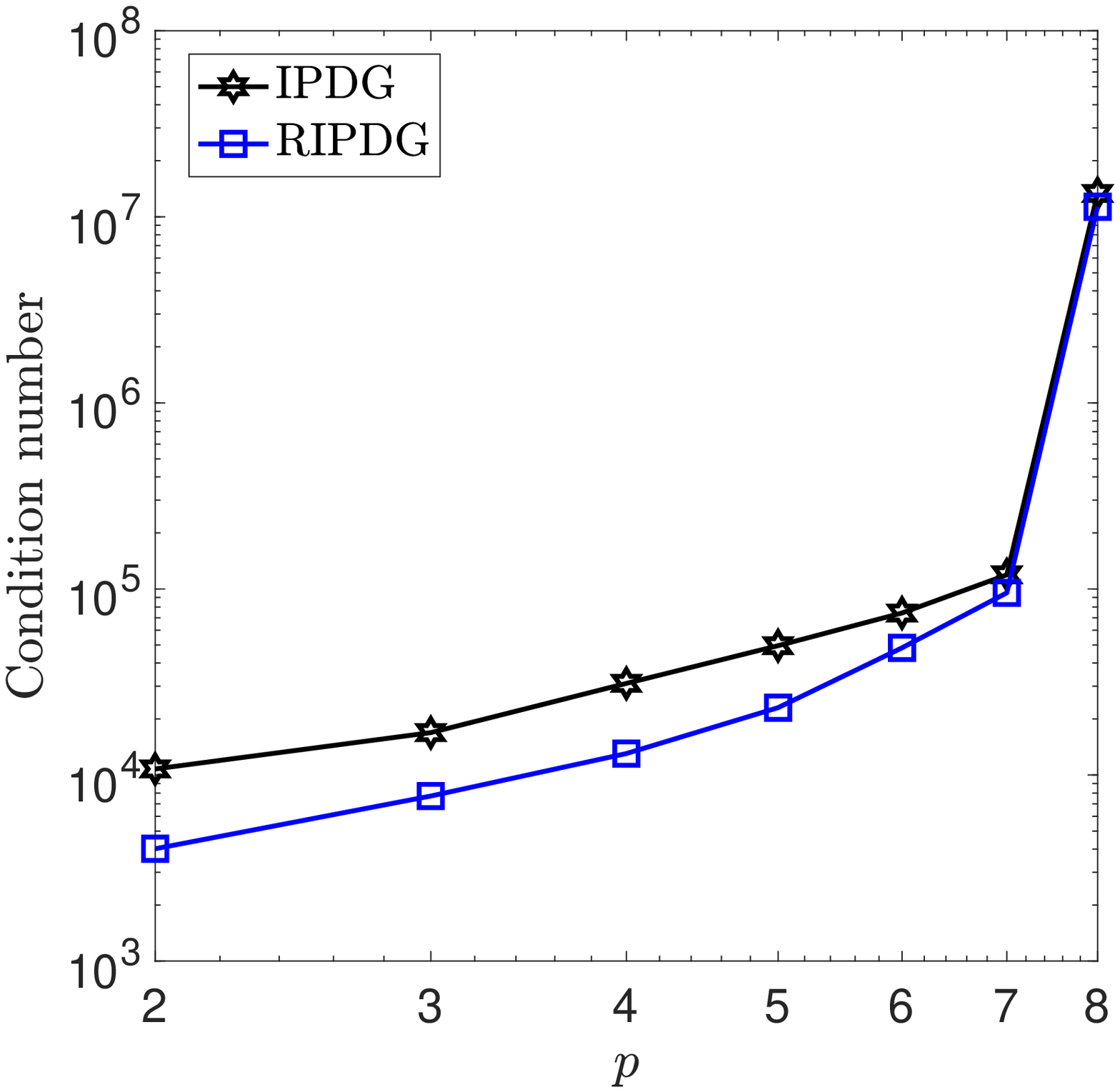}
	\caption{Example 1. ``Zigzag'' mesh. Maximum of the penalty parameter (left) and the condition number of the linear system (right)  for $\epsilon=10^{-3}$ and $p=1,\dots,8$.}\label{fig:ex1_condition_no_sigma_zigzag}
\end{figure}

	\begin{figure}[t]
	\includegraphics[height=5.5cm,width=6cm]{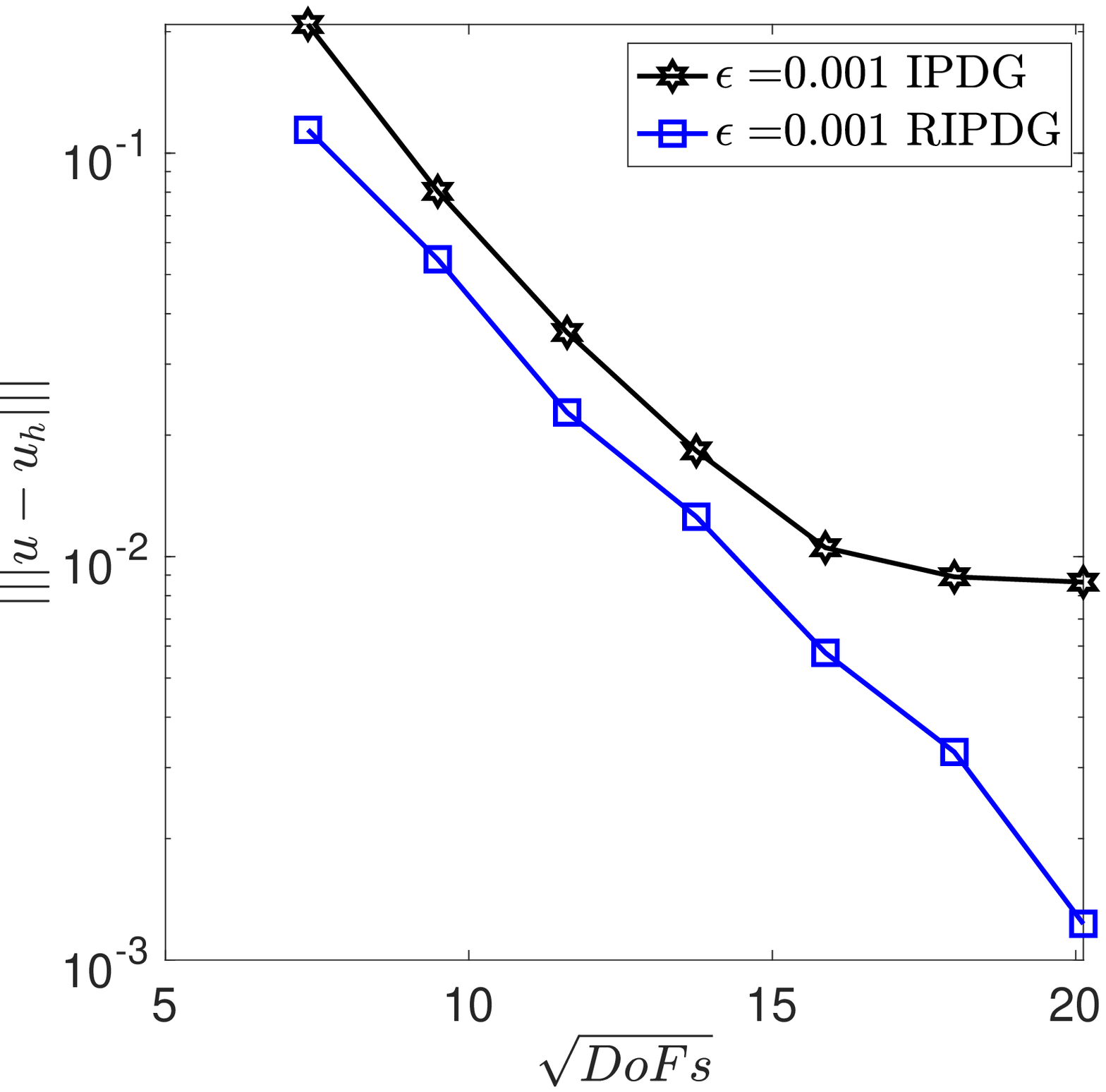}
	\includegraphics[height=5.5cm,width=6cm]{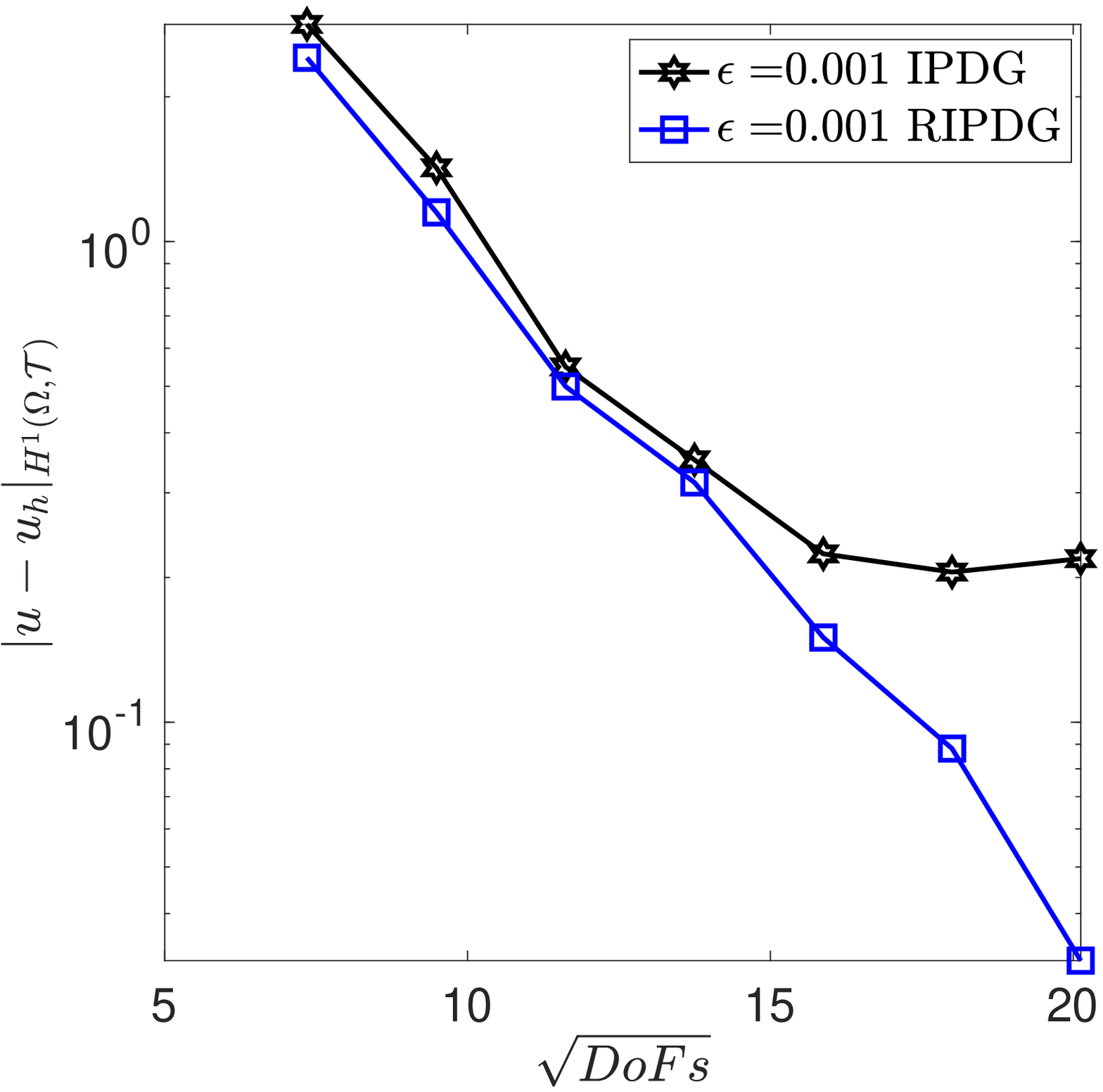}
	\caption{Example 1. ``Zigzag'' mesh. Convergence in dG-norm (left),  broken $H^1$-seminorm (right), 
		for $\epsilon=10^{-3}$ and $p=1,\dots,8$.}\label{fig:ex1_error_zigzag}
\end{figure}

\subsection{Example 2: Poisson problem with Gaussian solutions.} We now test the behaviour of each method against rapidly changing local polynomial degree $p$.  To that end, in $\Omega: =(-1,1)^2$, we consider the Poisson problem, with $a=I_{2\times 2}$ the $2\times 2$-identity matrix, with Dirichlet boundary conditions, and we choose $f$ such that the exact solution $u$ is given by
\begin{equation}
u(x,y): = \exp(-\alpha (x^2+y^2)),
\end{equation}
for some $\alpha>0$, which is analytic; in Figure \ref{ex2} (left) $u$ for $\alpha=100$ is depicted.

\begin{figure}
	\includegraphics[height=5.5cm,width=6cm]{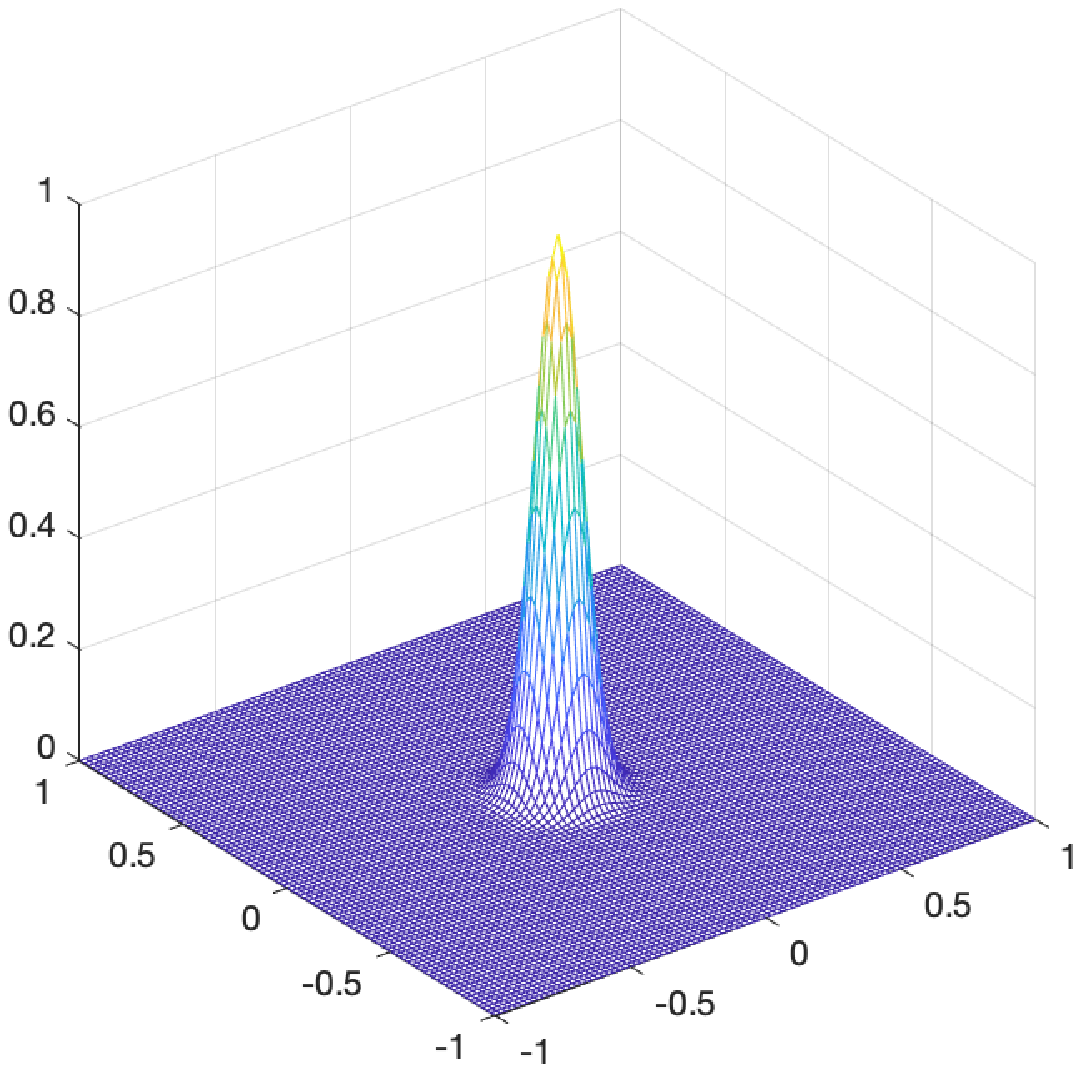}
	\includegraphics[height=5cm,width=5cm]{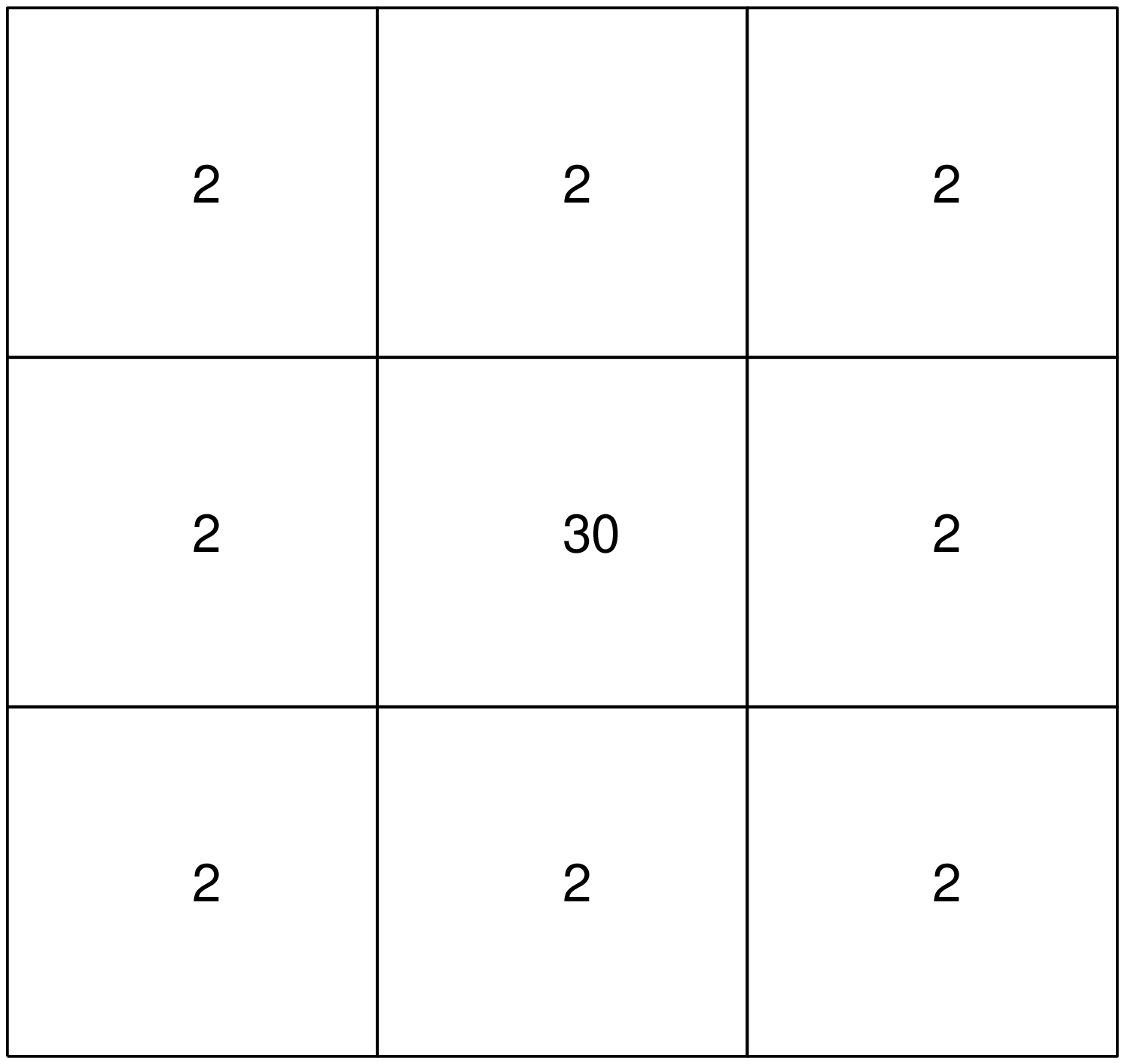}
	\caption{Example 2.  $u$ for $\alpha=100$ (left) and mesh and local polynomial degree distribution for the results of Table \ref{ex2:Comparison for dG methods} (right).}
	\label{ex2}
\end{figure}

For the first numerical experiment, we employ a $9$-element uniform square mesh polynomial degree $p=2$ in the eight boundary elements and polynomial degree $p=30$ on the interior $\tilde{K}:=(-1/3,1/3)^2$; this approximation space is presented in Figure \ref{ex2} (right). 
In Table \ref{ex2:Comparison for dG methods},  we present the comparison  between IPDG and RIPDG in terms of the maximum penalty used, condition number of the respective stiffness matrices and three different error measures. As we can see, RIPDG is outperforming in all respects, with the most striking improvement being in the stiffness matrix condition number. 

\begin{table}[h] 
	\begin{center}
		\begin{tabular}{||c|c|c|c|c|c||} 					
			\hline   
			method &  	 
			$\max_{F\subset\Gamma}^{}\sigma_F$& condition no.& $\|u-u_h\|$&$|u-u_h|_{H^1(\Omega,\mathcal{T})}$& 	 
			$|||u-u_h|||$   \\
			\hline											
			IPDG &  	 
			5.5800e+03& 5.1229e+06& 7.1923e-06& 1.9711e-04 & 	 
			2.1681e-04  \\
			\hline											
			RIPDG  &  	 
			61.69& 5.1148e+05&  6.5842e-06& 1.9169e-04 &  2.1305e-04 \\
			\hline		
		\end{tabular}  	
		\vspace{.2cm}
	\end{center}
	\caption{Example 2. $\alpha=100$. Comparison on the mesh and polynomial degree distribution given in Figure \ref{ex2} (right).}
	\label{ex2:Comparison for dG methods}
		\vspace{-.2cm}
\end{table}

The significant difference in the maximum of the penalty parameter confirms the respective discussion in Section \ref{RIPDG}. As a consequence, we can see that IPDG's condition number is about \textcolor{black}{$10$ times larger than RIPDG's for exactly} the same problem and approximation space.

Finally, we demonstrate the qualitative behaviour between the two methods resulting from the significant differences in the penalization magnitude. To that end, we consider the same problem with $\alpha=10$, resulting to a significantly less sharp Gaussian bump at the centre of the domain. This problem is approximated by the mesh given in Figure \ref{fig:ex2_fine}, with each elemental polynomial degree also given.
\begin{figure}
	\includegraphics[height=4.5cm,width=4.5cm]{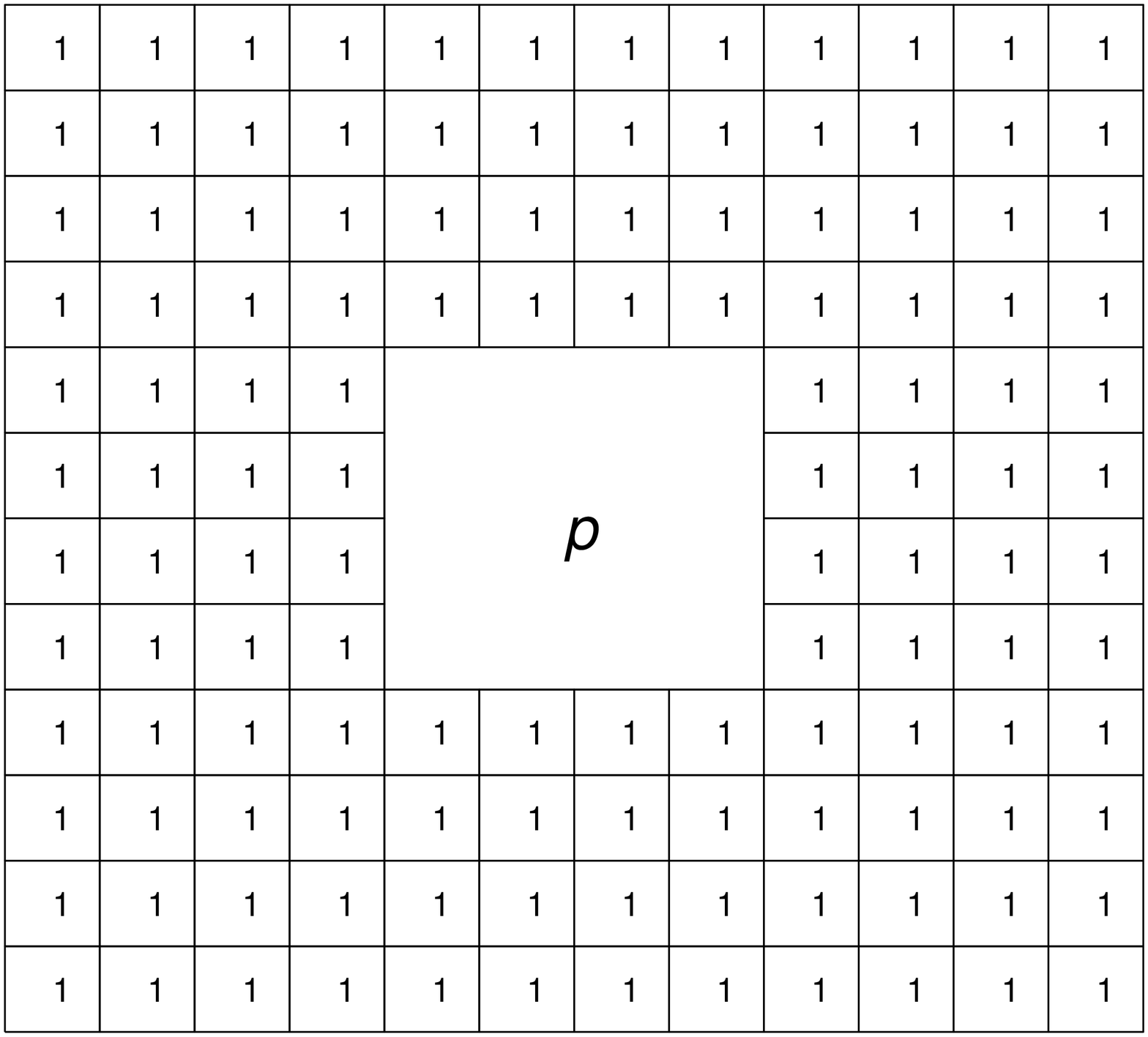}
	\caption{Mesh used for the solutions in Figure \ref{qual} with $p=1$ on all `small' elements and $p=5,8$ in the central `large' element. }\label{fig:ex2_fine}
\end{figure}
In Figure \ref{qual} we provide the solution profiles for RIPDG (left) and IPDG (right) for $\alpha=10$, $p=1$ on all `small' elements and $p=3$ (first line), $p=5$ (second line), and $p=8$ (last line) in the central `large' element. We observe significantly different behaviour in the `large' central element interfaces, due to the substantially smaller penalization required by the RIPDG method. Interestingly, however, we also observe modestly different behaviour further away from the `large' element interfaces also, especially for $p=3$.
\begin{figure}
		\includegraphics[height=5cm,width=6cm]{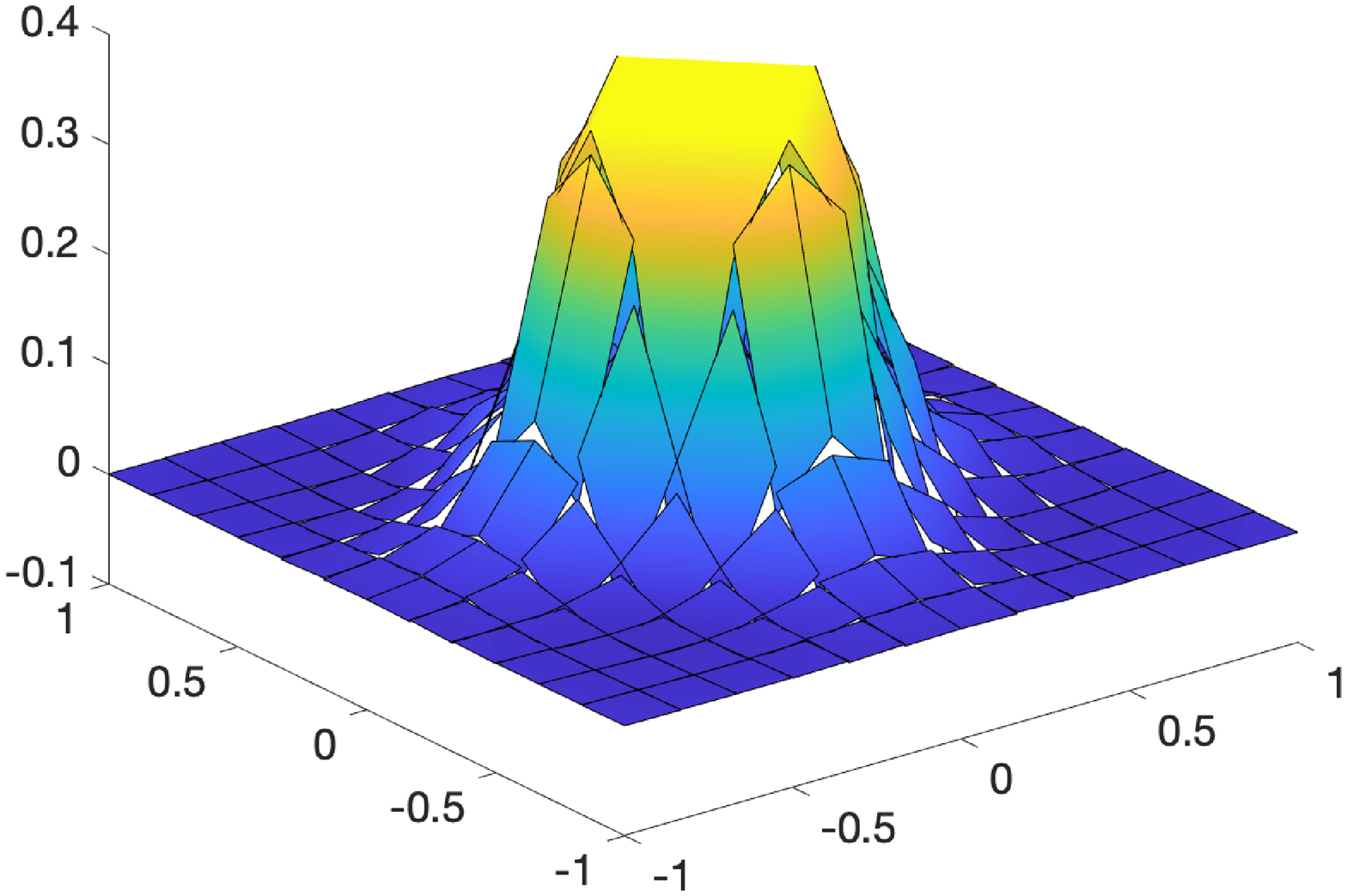}	\includegraphics[height=5cm,width=6cm]{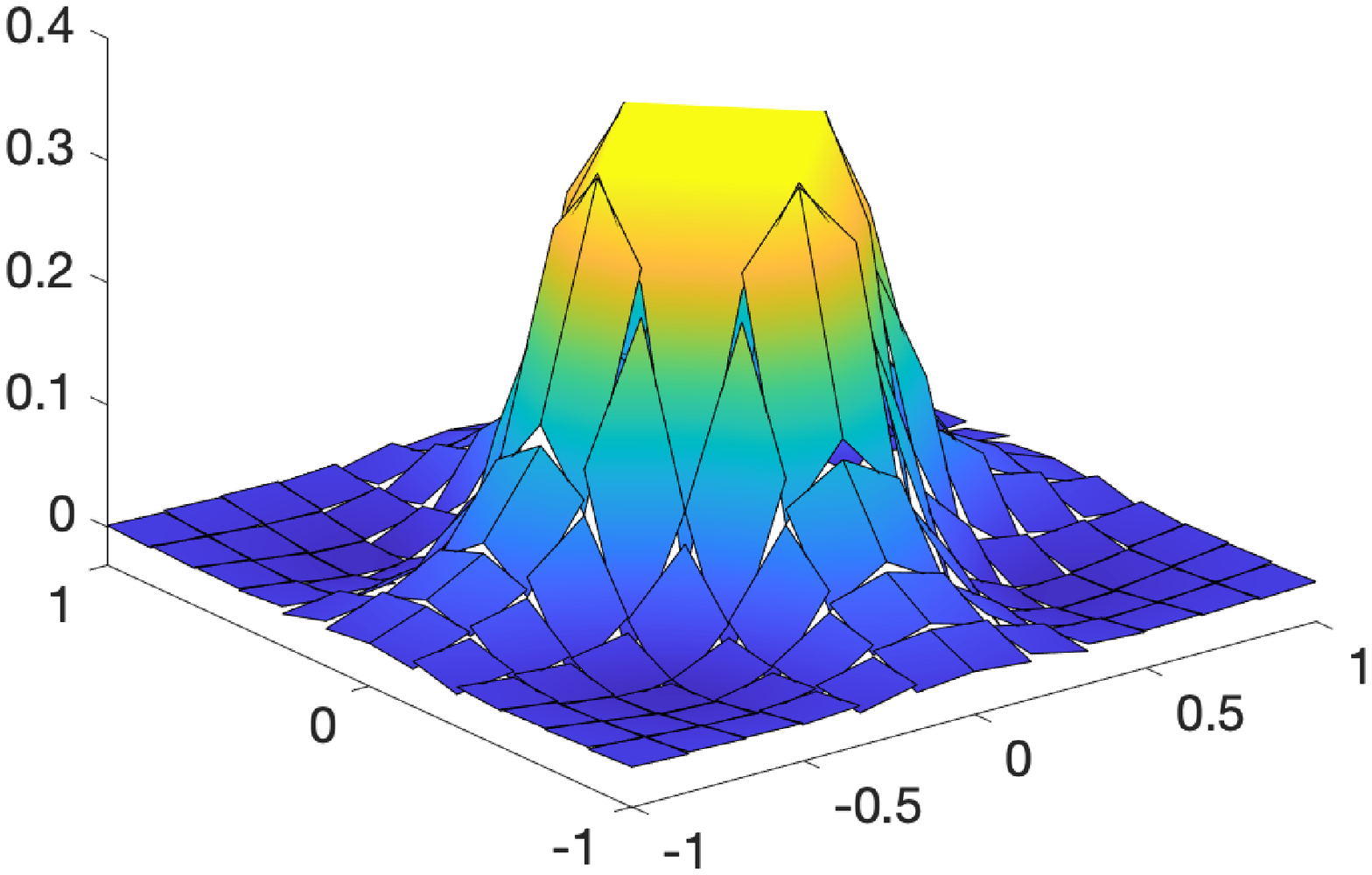}\\
	\includegraphics[height=5cm,width=6cm]{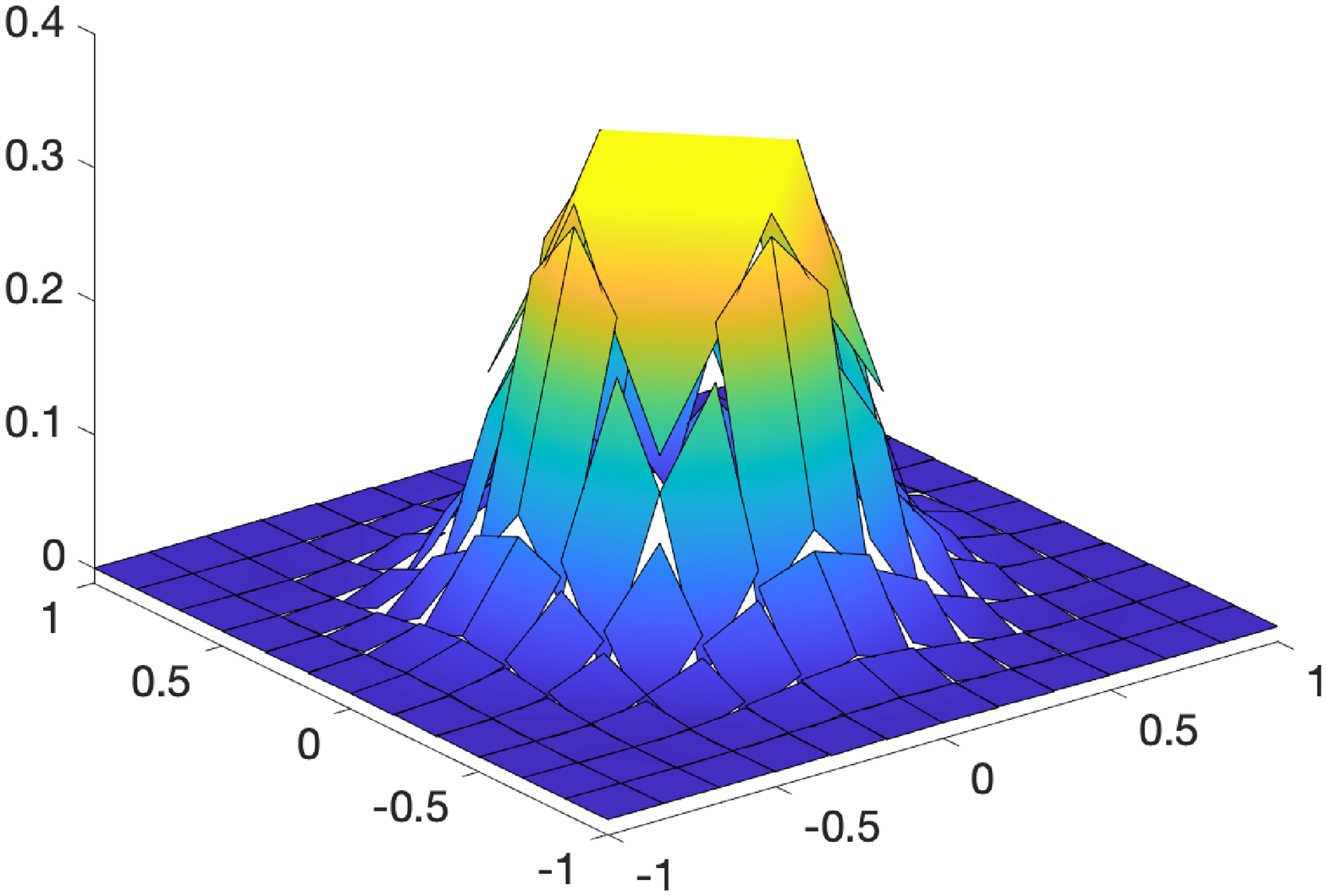}	\includegraphics[height=4cm,width=6cm]{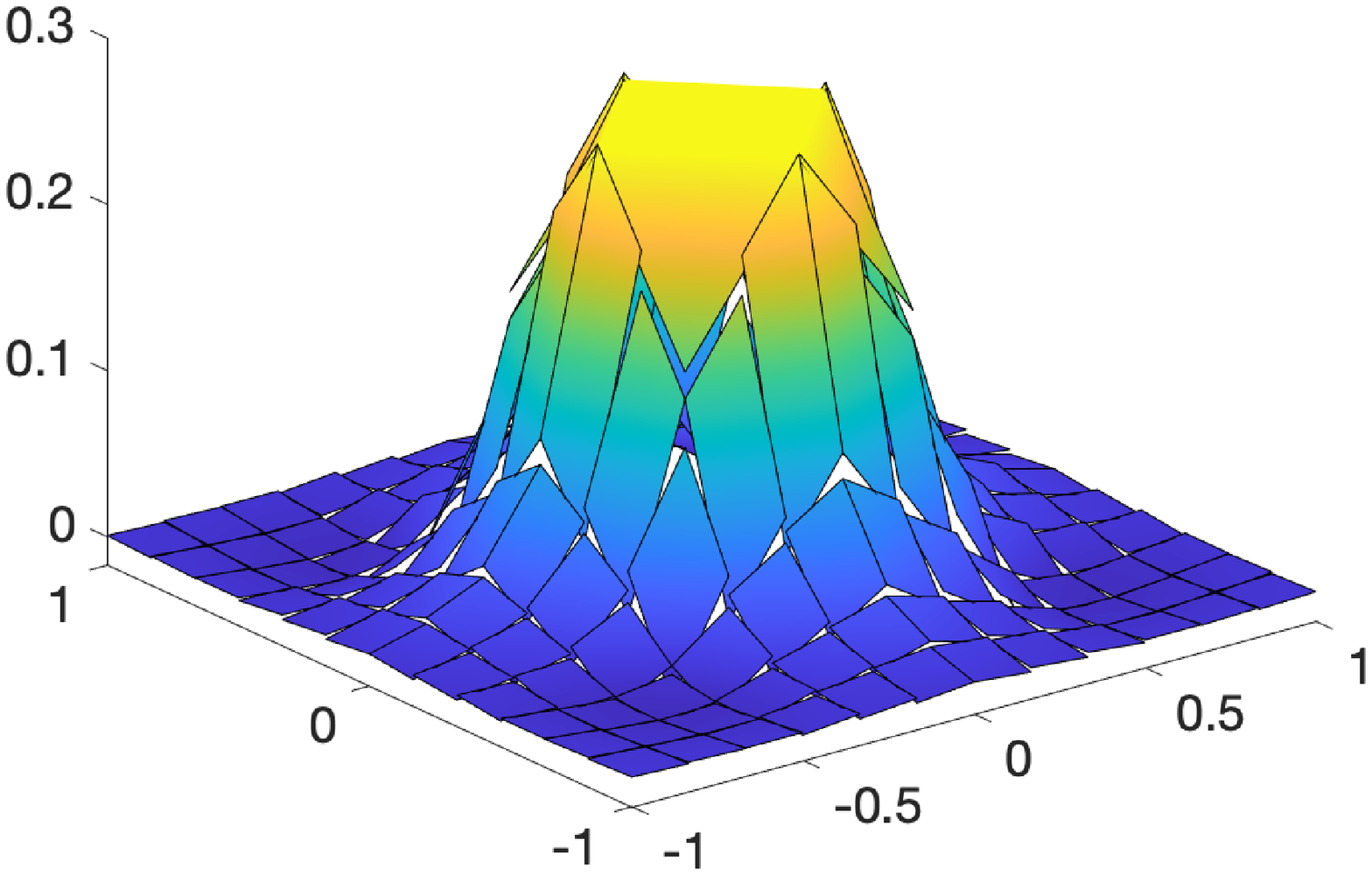}\vspace{.4cm}\\
		\includegraphics[height=4cm,width=6cm]{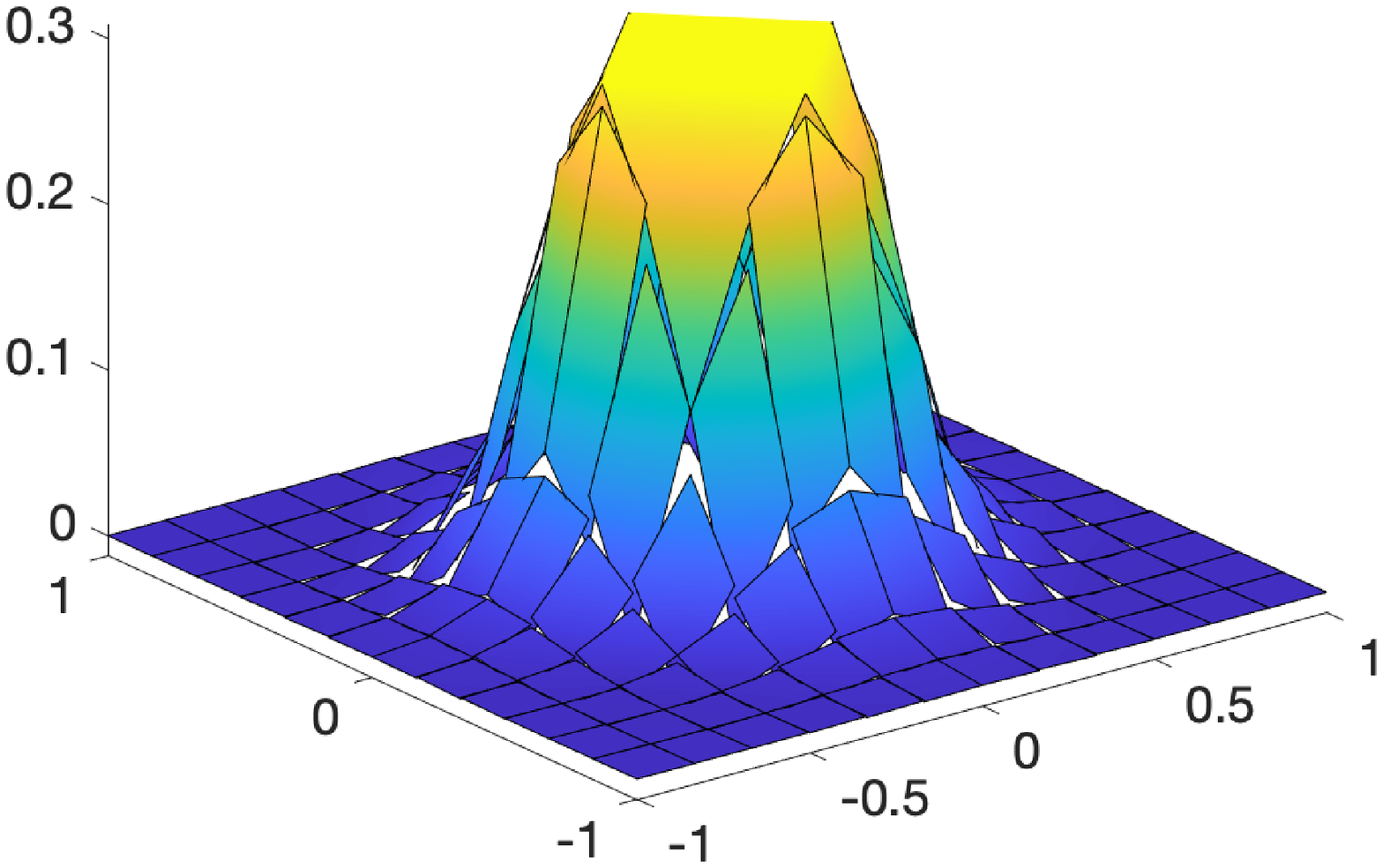}	\includegraphics[height=4cm,width=6cm]{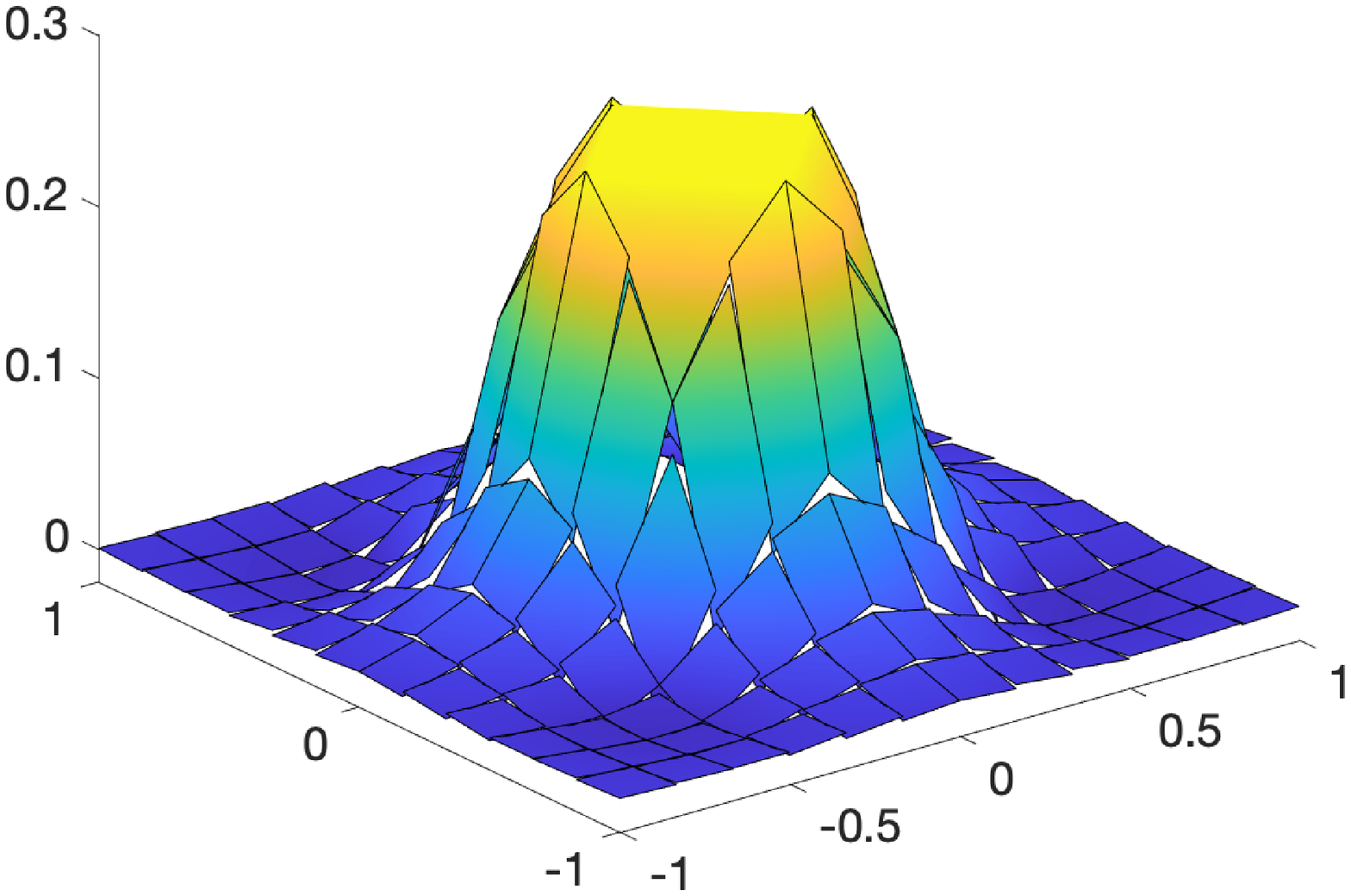}
		\caption{Solution profiles for RIPDG (left) and IPDG (right) for $\alpha=10$, $p=1$ on all `small' elements and $p=3$ (first line), $p=5$ (second line), and $p=8$ (last line) in the central `large' element. }\label{qual}
\end{figure}

%
%
%
%
%

 Of course, one may not use encounter or use approximations spaces as the ones considered in this example. Nonetheless, such mesh and polynomial degree distributions may arise locally in the context of $hp$-adaptive procedures. In any case, this study is revealing in how extreme approximation space local variation behaviour may result into possibly inferior performance of classical approaches. 


\Rev{
\subsection{Example 3. Highly agglomerated meshes.}
We now consider the Poisson problem with Dirichlet boundary conditions admitting the smooth solution $u:=\sin(\pi x)\sin(\pi y)$ on highly agglomerated meshes: we start from a fine simplicial mesh with $131,072$ triangles which is further agglomerated in a random fashion using standard partitioning tools into $37$ complicated polygonal elements shown in Figure  \ref{fig:ex4_mesh}. The discontinuity penalization parameter is selected as described briefly in Section \ref{poly_RIPDG}; we refer to \cite{dg-ease} for details.  On this fixed mesh, we compute the numerical solution for each method for uniform $p=1,\dots,5$; these are presented in Figure \ref{fig:ex4_error}. As expected, the errors in all norms decay exponentially. Even though the polynomial degree is uniform and neighbouring elements are of `similar' sizes, we can still observe an improvement in the errors for RIPDG against IPDG of about 25\%. In Figure \ref{fig:ex4_condition_no_sigma_wrong_mesh},  we can see that the maximum penalty and condition numbers for IPDG compared to RIPDG are approximately $4$ and $2.5$ times larger, respectively.
}

\begin{figure}
	\includegraphics[height=5.5cm,width=5.5cm]{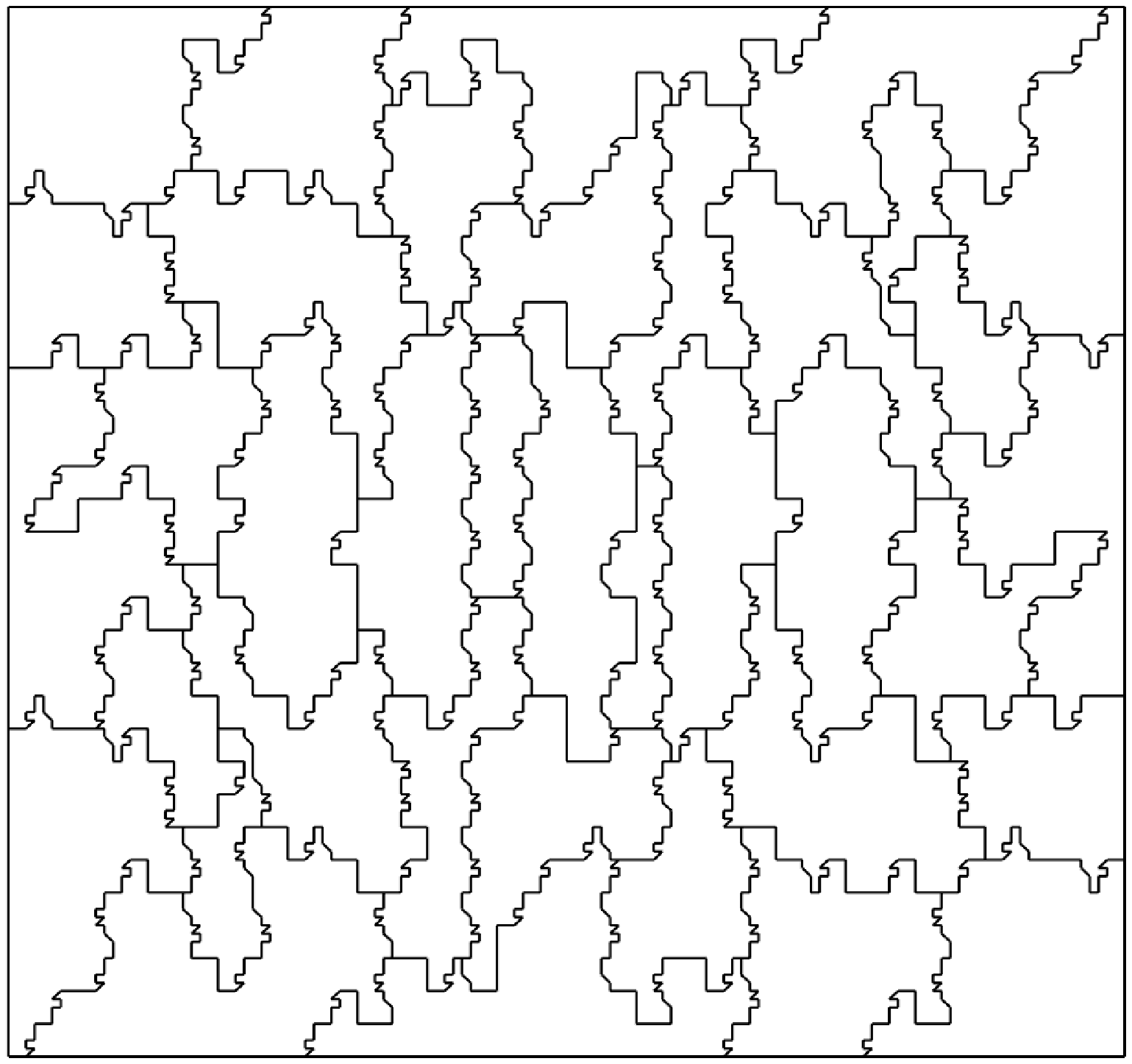}
	\caption{$37$ agglomerated elements made by $131072$ triangular meshes}\label{fig:ex4_mesh}
\end{figure}

\begin{figure}
	\includegraphics[height=5.5cm,width=6cm]{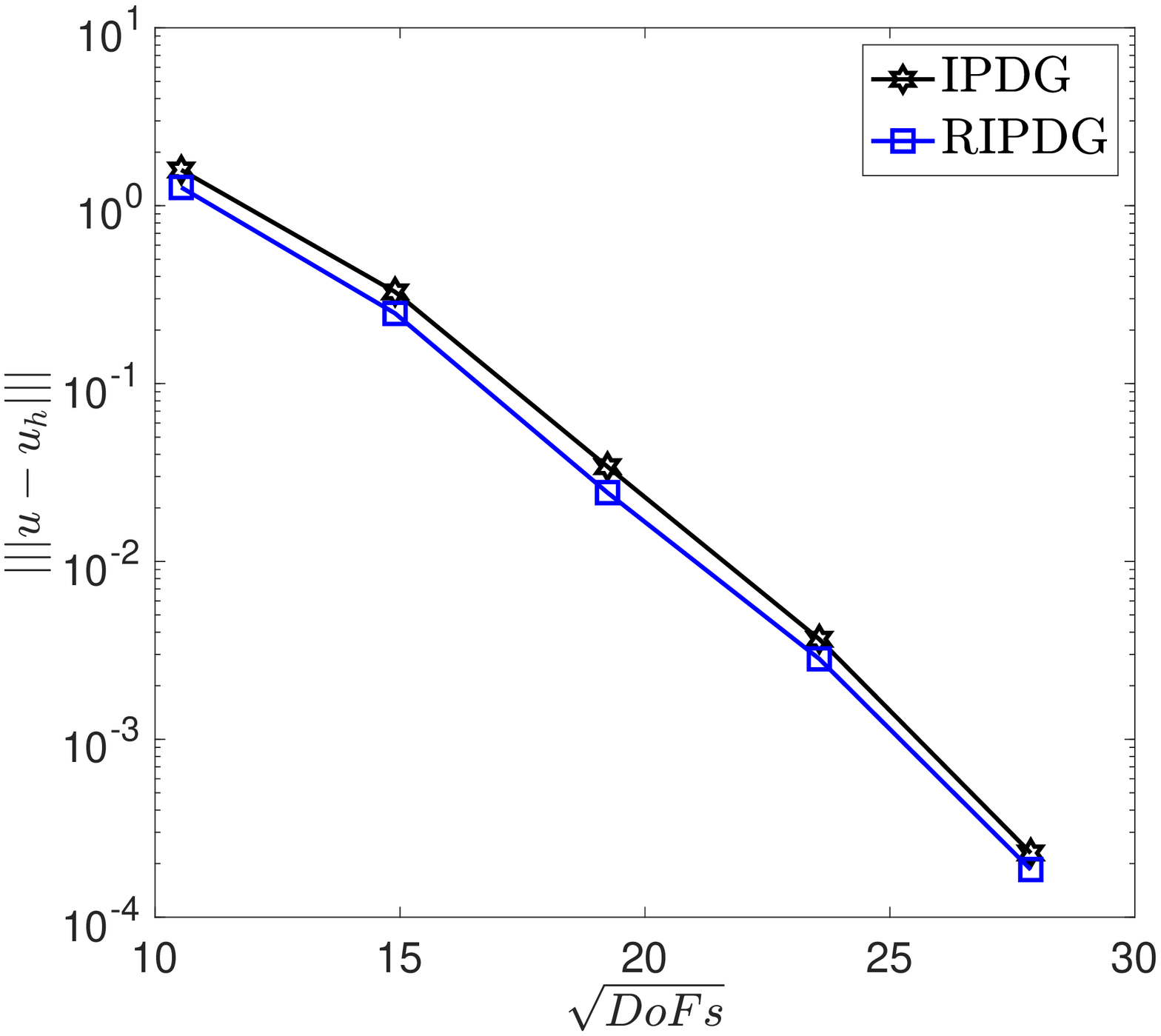}
	\includegraphics[height=5.5cm,width=6cm]{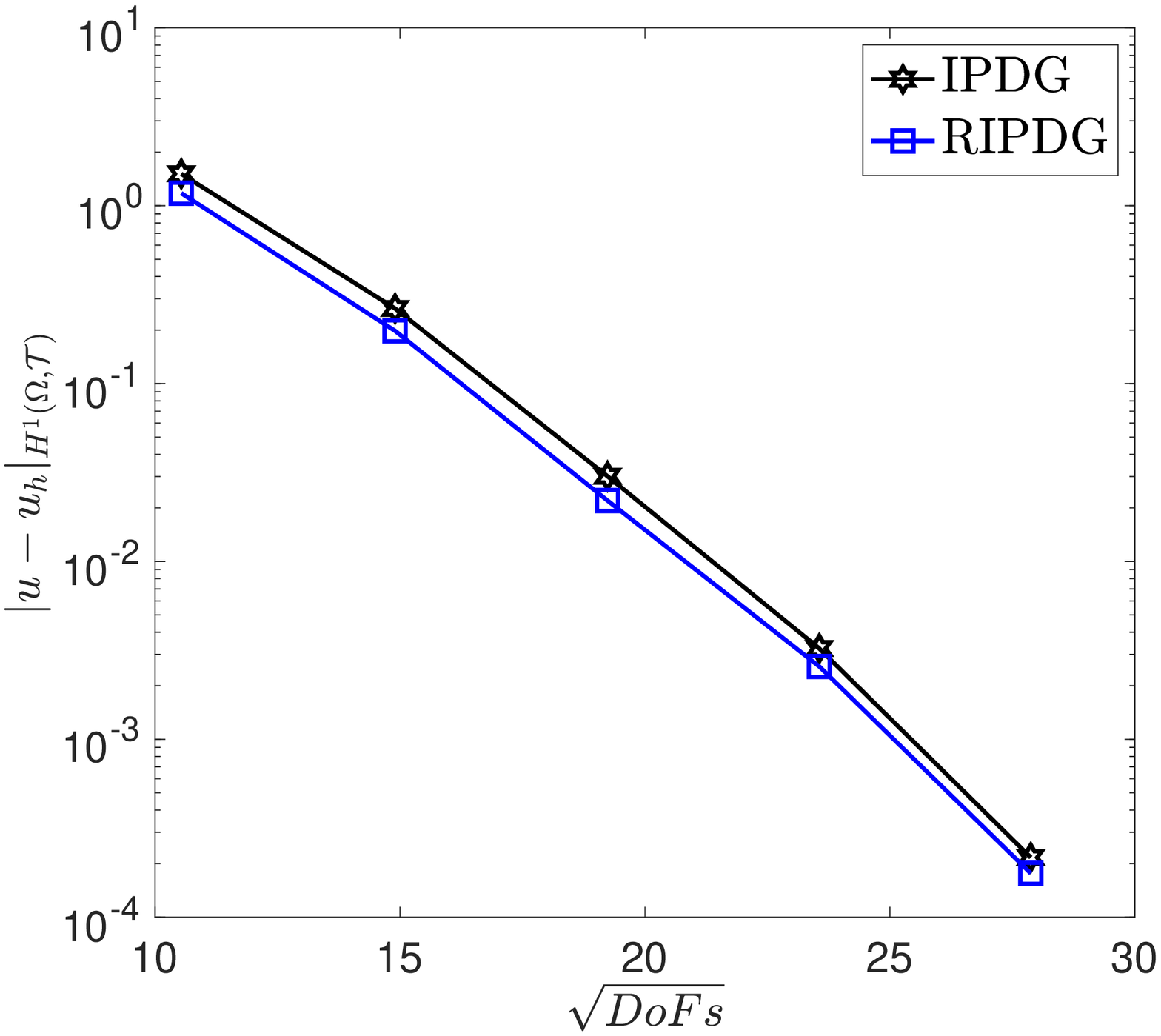}
	\caption{Exponential convergence in dG--norm (Left)  and broken $H^1$--norm (Right)  with $p=1,\dots,5$}\label{fig:ex4_error}
\end{figure}

\begin{figure}
	\includegraphics[height=5.5cm,width=6cm]{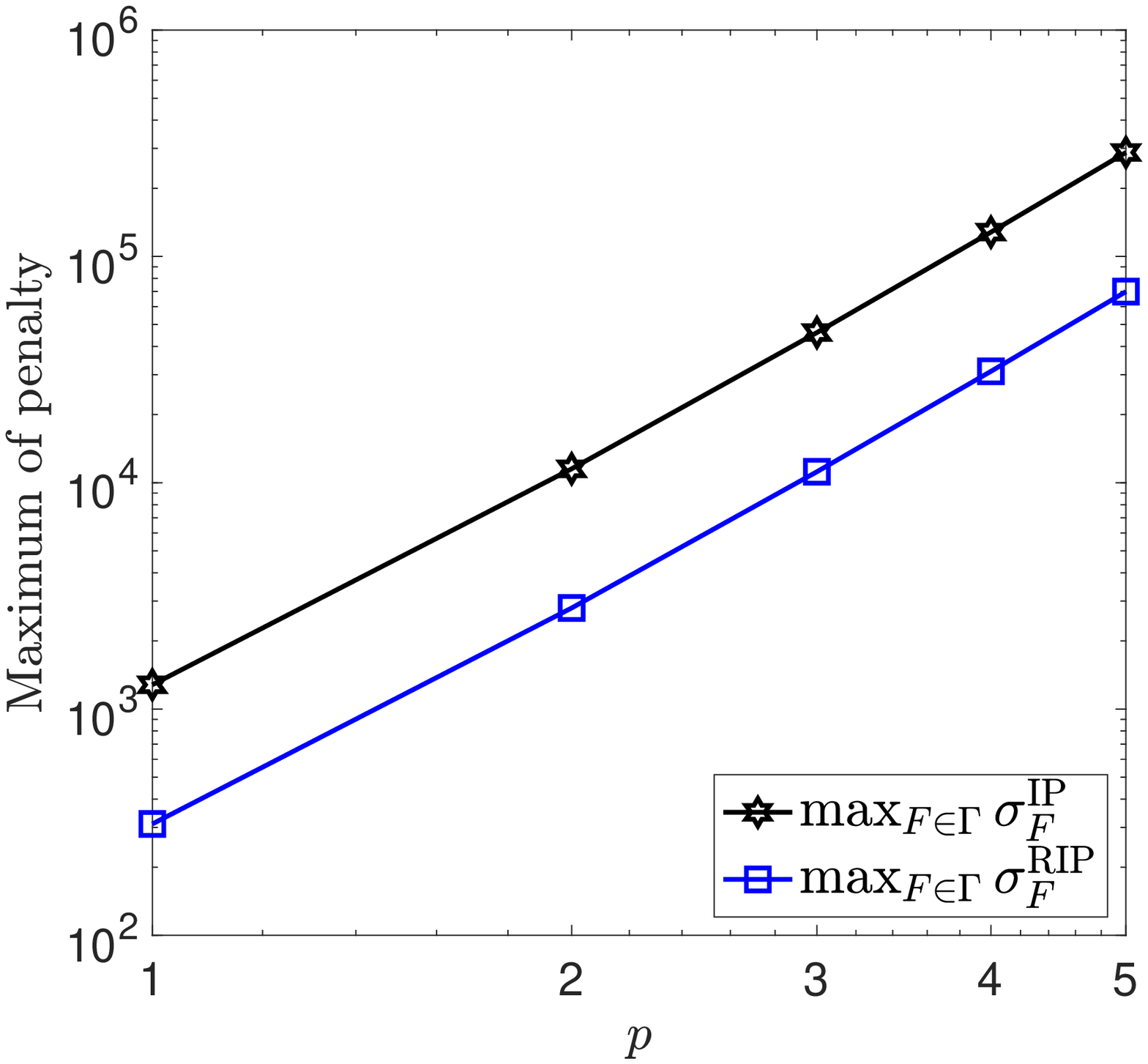}
	\includegraphics[height=5.5cm,width=6cm]{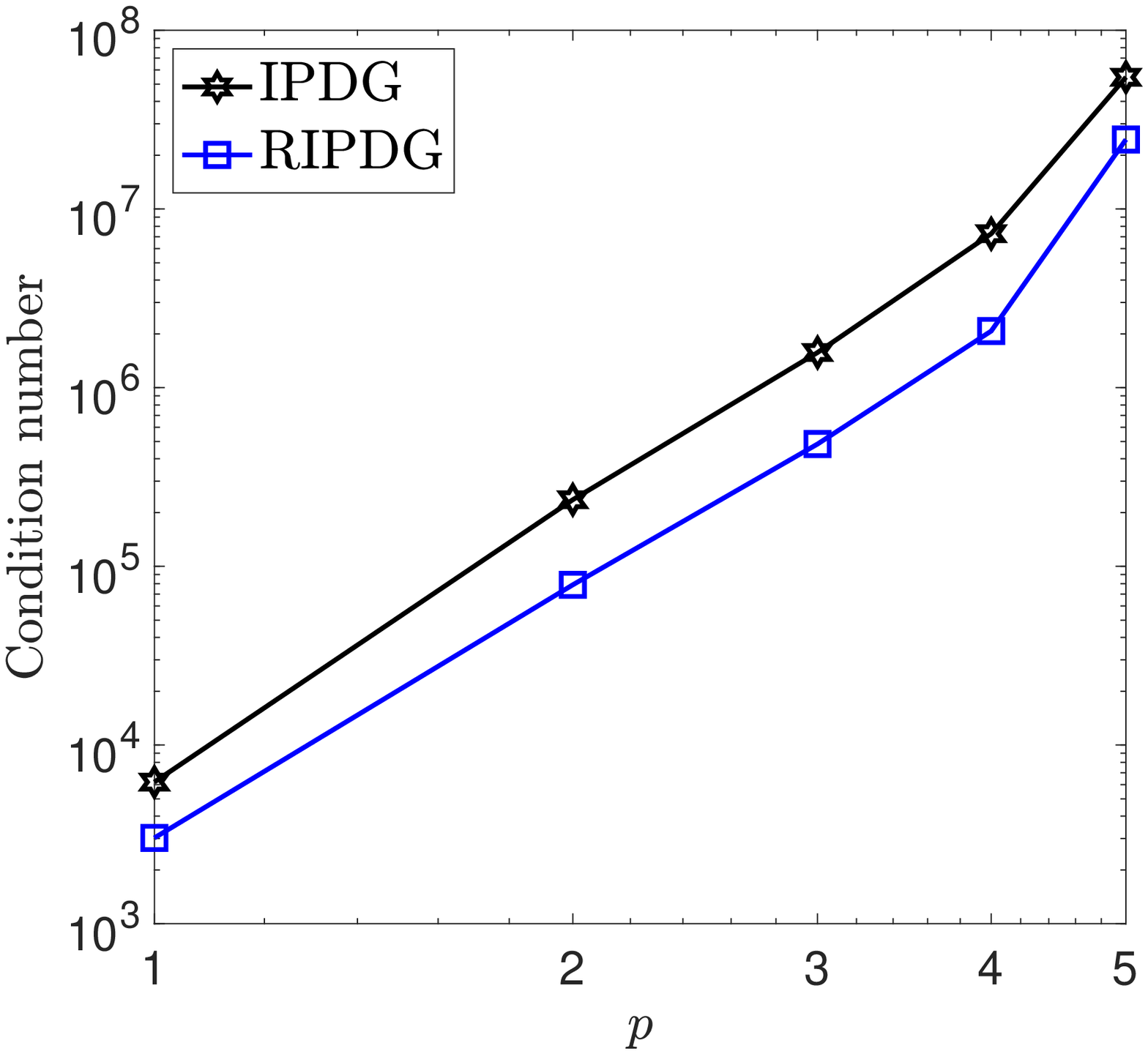}
	\caption{The maximum penalty parameters (Left)  and the condition number of the linear system(Right)  $p=1,\dots,5$}\label{fig:ex4_condition_no_sigma_wrong_mesh}
\end{figure}

	\bibliographystyle{siam}
	\bibliography{literature}
	
\end{document}